\documentclass[12pt,a4paper,reqno]{amsart}

\usepackage{amssymb}
\usepackage{enumerate}

\newtheorem{theorem}{Theorem}[section]
\newtheorem{proposition}[theorem]{Proposition}
\newtheorem{lemma}[theorem]{Lemma}
\newtheorem{corollary}[theorem]{Corollary}

\theoremstyle{definition}
\newtheorem{definition}[theorem]{Definition}
\newtheorem{notation}[theorem]{Notation}
\newtheorem{remark}[theorem]{Remark}

\newtheorem{example}[theorem]{Example}

\newcommand{\suchthat}{\;\ifnum\currentgrouptype=16 \middle\fi|\;}

\newcommand{\p}{\mathbb{P}}
\renewcommand{\k}{\mathrm{k}}
\newcommand{\Aut}{\mathrm{Aut}}
\newcommand{\End}{\mathrm{End}}

\newcommand{\GL}{\mathrm{GL}}
\newcommand{\SL}{\mathrm{SL}}
\newcommand{\SAut}{\mathrm{SAut}}
\newcommand{\SJ}{\mathrm{SJ}}
\newcommand{\J}{\mathrm{J}}
\newcommand{\car}{\mathrm{char}}

\newcommand{\Ker}{\mathrm{Ker}}
\renewcommand{\Im}{\mathrm{Im}}
\newcommand{\m}{\mathfrak{m}}

\newcommand{\SAff}{\mathrm{SAff}}
\newcommand{\Aff}{\mathrm{Aff}}
\newcommand{\A}{\mathbb{A}}
\newcommand{\Jac}{\mathrm{Jac}}
\newcommand{\C}{\mathbb{C}}
\newcommand{\N}{\mathbb{N}}
\newcommand{\Z}{\mathbb{Z}}

\DeclareMathOperator{\Bir}{Bir}

  \usepackage[all]{xy}
\xyoption{matrix}

\title{Conjugacy classes of special automorphisms of the affine spaces}
\author{J\'er\'emy Blanc}
\address{Mathematisches Institut \\ 
Universit\"at Basel \\
Rheinsprung 21 \\
CH-4051 Basel \\
Switzerland}
\email{jeremy.blanc@unibas.ch}\thanks{The author gratefully acknowledges support by the Swiss National Science Foundation Grant  ``Birational Geometry'' PP00P2\_153026 /1 and by the French National Research Agency Grant ``BirPol'', ANR-11-JS01-004-01.}
\date{\today}

\begin{document}
\begin{abstract}
In the group of polynomial automorphisms of the plane, the conjugacy class of an element is closed if and only if the element is diagonalisable. In this article, we show that this does not hold for the group of special automorphisms, giving then a first step towards the direction of showing that this group is not simple, as an infinite-dimensional algebraic group.
\end{abstract}
\maketitle
\section{Introduction}
In this article, $\k$ will always denote an algebraically closed field.
The conjugacy classes of the algebraic groups $\GL(n,\k)$ and $\SL(n,\k)$ are well-known. In particular, the following observation is classical:

{\it An element is diagonalisable if and only if its conjugacy class is Zariski-closed.}

As observed in \cite{MauFur}, the same holds for the group $\Aut(\A^2_\C)$ of complex polynomial automorphisms of the affine plane. Here, the topology corresponds to the topology of $\Aut(\A^n_\k)$ induced by families parametrised by algebraic varieties $A$, called morphisms $A\to \Aut(\A^n_\k)$ and corresponding to elements of $\Aut(\A^n_{\k[A]})$ (see $\S\ref{SubSec:Topology}$).

In fact, there is one easy direction in the result of \cite{MauFur}, which corresponds to show that if the conjugacy class is closed, then the element is diagonalisable. This works over any algebraically closed field $\k$ and follows from the following observation: If $f\in \Aut(\A^n_\k)$ is an element that fix the origin, the conjugation of $f$ by $$(x_1,\dots,x_n)\to \left(\frac{1}{t} x_1,\dots,\frac{1}{t}x_n\right)$$
yields an element of $\Aut(\A^n_{\k(t)})$ whose value at $t=0$ is the linear part of $f$, which is an element of $\GL(n,\k)$. Moreover elements of $\Aut(\A^2_\k)$ which do not have fix points are easy to handle (these are conjugate to $(x_1,x_2)\mapsto (x_1+1,ax_2+P(x_2))$ for some polynomial $P\in \k[x_2]$ and $a\in \k^*$).

In this article, we focus on the closed normal subgroup $\SAut(\A^n_\k)$ of $\Aut(\A^n_\k)$ of elements of Jacobian $1$. We will show that the conjugacy classes of the two groups $\SAut(\A^n_\k)$ and $\Aut(\A^n_\k)$ have a very different behaviour. 

We say that an element $f\in\Aut(\A^n_\k)$ is \emph{dynamically regular} if the extensions of $f$ and $f^{-1}$ to $\p^n_\k$ have disjoint indeterminacy loci (see $\S\ref{SubSec:Dynamical}$). We will also say that $f$ is \emph{algebraic} if $\{\deg(f^n)\}_{n\in \mathbb{N}}$ is bounded. The dynamically regular elements are never algebraic, and in dimension~$2$, non-algebraic elements are conjugate to dynamically regular elements (see Remark~\ref{Rem:AlgebraicSJ}).

The first result that we obtain is to show that there is no degeneration of conjugates of dynamically regular elements in $\SAut(\A^n_\k)$, contrary to the case of $\Aut(\A^n_\k)$:

\begin{theorem}\label{Thm:Regular}
Let $f\in \SAut(\A^n_\k)$ be a dynamically regular element.
\begin{enumerate}[$(1)$]
\item
If $\alpha\in \SAut(\A^n_{\k((t))})$ is such that $\alpha f \alpha^{-1}$ has a value at $t=0$, then this value is conjugate to $f$ by $\alpha(0)$ in $\SAut(\A^n_\k)$ $($in particular $\alpha$ is defined at $t=0)$. 
\item
For each integer $d$, the set $\{gfg^{-1}\mid g\in \SAut(\A^n_\k), \deg(g)\le d\}$ is closed.
\item
If $n=2$, the conjugacy class of $f$ in $\SAut(\A^n_\k)$ is closed.
\item
If $\k$ is uncountable, the following holds: For each morphism $A\to \SAut(\A^n_{\k})$, where $A$ is an algebraic variety, the preimage of the conjugacy class of $f$ contains the closure of each locally closed subset $B\subset A$ that it contains.
\end{enumerate}
\end{theorem}
\begin{remark} The proof of this result is given in Section~$\ref{Sec:ConjRegular}$. As we will show, part $(1)$ imply the others.

In the notation of \cite{FurKra}, Assertion $(4)$ can be reinterpreted by saying that the conjugacy class $C(f)$ of a dynamically regular element $f$ is \emph{weakly-closed} in $\SAut(\A^n_\k)$.
\end{remark}
In dimension $2$, an easy consequence of Theorem~\ref{Thm:Regular} and of the Jung-Van der Kulk theorem is the fact the conjugacy class of non-algebraic elements of $\SAut(\A^2_\k)$ is in fact closed, contrary to the case of $\Aut(\A^2_\k)$. In fact, we can be really more precise: we describe in Section \ref{Sec:ConjclassDim2} the conjugacy classes of elements in $\SAut(\A^2_\k)$, and decide which one is closed. In particular, we obtain the following complete description:

\begin{theorem}\label{Thm:Dimension2}
Let $f\in \SAut(\A^2_\k)$. Then, one of the following holds:
\begin{enumerate}[$(1)$]
\item
If $f$ is diagonalisable, its conjugacy class is closed.
\item
If $f$ is algebraic but not diagonalisable, its conjugacy class is not closed. More precisely, there exists an element $F\in \SAut(\A^2_{\k[t]})$ such that for each $t\not=0$, $F(t)\in \SAut(\A^2_K)$ is conjugate to $f$, and $F(0)\in \SAut(\A^2_\k)$ is diagonalisable.
\item
If $f$ is not algebraic, its conjugacy class is closed.
\end{enumerate}
\end{theorem}
\begin{remark}
The set $\SAut(\A^2_\k)_{\mathrm{alg}}$of algebraic elements of $\SAut(\A^2_\k)$ being closed (Corollary~\ref{Coro:Algclosed}), we can decompose $\SAut(\A^2_\k)$ as a infinite union of disjoint closed sets, one being $\SAut(\A^2_\k)_{\mathrm{alg}}$ and the others being conjugacy classes of non-algebraic elements. The group $\SAut(\A^2_\k)$ is however irreducible, by the simple observation made above.
\end{remark}
\begin{remark}
Note that these results show that the group $\SAut(\A^2_\k)$ is more rigid than the group $\Aut(\A^2_\k)$, in the sense that there are less possible degenerations of conjugates. 

One can check, using the conjugations around fixed points as above, that every normal closed subgroup of $\Aut(\A^2_\k)$ is either trivial, $\SAut(\A^2_\k)$ or $\Aut(\A^2_\k)$. The interesting question is then to know whether $\SAut(\A^2_\k)$ contains non-trivial closed normal subgroups (it contains many non-trivial normal subgroups by \cite{FurLam}, which contains only non-algebraic elements, except the identity).

The fact that the conjugacy classes of non-algebraic elements is closed goes towards the direction of indicating that $\SAut(\A^2_\k)$ could contain non-trivial closed normal subgroups. This text can then be viewed as a first step towards the study of the simplicity of $\SAut(\A^n_\k)$, viewed as an infinite-dimensional algebraic group (ind-group). In \cite{Sha1}, it is claimed that this one is simple, but the proof contains serious gaps. 
\end{remark}

The author thanks Jean-Philippe Furter and Hanspeter Kraft for interesting discussions during the preparation of this article and for giving him access to the preprint \cite{FurKra}.

\section{Preliminaries}
As we said, in the sequel $\k$ will always be an algebraically closed field. We will sometimes also work on a general field (most of the time with an extension of $\k$), and will denote  it by $K$.

In this section, we introduce the terminology and give some basic results (most of them classical, maybe in other terms) on the topology of $\Aut(\A^n_\k)$ ($\S\ref{SubSec:Topology}$), the relation between the iterations of a map and the indeterminacy sets at infinity ($\S\ref{SubSec:Dynamical}$) and the families of automorphisms parametrised by formal series ($\S\ref{SubSec:Families}$), that we will need in the next sections (Sections~\ref{Sec:ConjRegular} and \ref{Sec:ConjclassDim2}).
\subsection{Topology on $\Aut(\A^n_\k)$}\label{SubSec:Topology}
\begin{notation}
Let $R$ be any commutative unitary ring. 

$(1)$ We denote by $\End(\A^n_R)$ the set of algebraic endomorphisms of $\A^n_R$. An element $f\in \End(\A^n_R)$ is given by
$$f\colon (x_1,\dots,x_n)\mapsto (f_1(x_1,\dots,x_n),\dots,f_n(x_1,\dots,x_n))$$
for some polynomials $f_1,\dots,f_n\in R[x_1,\dots,x_n]$. The degree of $f$ is by definition the maximal degree of the $f_i$, and we will use the notation
$$f=(f_1,\dots,f_n).$$
This corresponds to a natural bijection $\End(\A^n_R)\to (R[x_1,\dots,x_n])^n.$

$(2)$ The group $\Aut(\A^n_R)$ is equal to the group of automorphisms of $\A^n_R$, i.e.~to the elements of $\End(\A^n_R)$ that admit an inverse in this set.

$(3)$ For each $f\in \End(\A^n_R)$, we denote by $\Jac(f)=\det \left(\frac{\partial f_i}{x_j}\right)_{i,j=1}^n\in R[x_1,\dots,x_n]$ the Jacobian of $f$, and denote by $\SAut(\A^n_R)$ the normal subgroup of $\Aut(\A^n_R)$ given by $\{f\in \Aut(\A^n_R)\mid \Jac(f)=1\}$.

$(4)$ We denote by $\End(\A^n_R)_{\le d}$ and $\Aut(\A^n_R)_{\le d}$ the subsets of $\End(\A^n_R)$ and $\Aut(\A^n_R)$ respectively, given by elements of degree $\le d$.
\end{notation}
\begin{example}\label{Example:Triangular}
For each $p_1\in R[x_1],p_2\in R[x_1,x_2],\dots,p_{n-1}\in R[x_1,\dots,x_{n-1}]$ and $a_1,\dots,a_n\in R^*$ the element
$$(a_1x_1,a_2x_2+p_1,a_3x_3+p_2,\dots,a_nx_n+p_{n-1})$$
belongs to $\Aut(\A^n_R)$. Such elements are usually called \emph{triangular}, or \emph{de Jonqui\`eres}.
\end{example}
\begin{remark}\label{Rem:ConjJac}
Suppose that $R$ is a field $K$. Extending the scalars to an algebraically closed field, we observe that  the Jacobian matrix of every element of $\Aut(\A^n_K)$ is invertible everywhere, so  $\Jac(f)\in K^{*}$. In particular, $$\Aut(\A^n_K)\subset \{f\in \End(\A^n_K)\mid \Jac(f)\in K^{*}\}$$ and the equality is the classical Jacobian conjecture, open for any $n\ge 2$.
\end{remark}

If $Z$ is an algebraic variety defined over $\k$, where $\k$ is algebraically closed as before, there is a natural way to endow the group $\Bir(X)$ of birational transformations of $Z$ with a topology (see for example \cite{De}, \cite{Se}, \cite{Bl} or \cite{BF}). When restricted to the subgroup $\Aut(Z)$ of automorphisms, we obtain the following:
\begin{definition}\label{Defi:MorphismAut}
Let $A,Z$ be two algebraic varieties defined over $\k$. We say that a \emph{morphism} $f\colon A\to \Aut(Z)$ is a map given by an $A$-automorphism of $A\times Z$.

The \emph{Zariski topology} on $\Aut(Z)$ is defined as follows: a set $F\subset \Aut(Z)$ is closed if and only if $f^{-1}(F)\subset A$ is closed for any algebraic variety $A$ and any morphism $f\colon A\to \Aut(Z)$.
\end{definition}
\begin{remark}
When the group $\Aut(Z)$ has a natural structure of an algebraic group (for example when $Z=\p^n$), the topology defined above agrees with the classical topology of the algebraic group, and morphisms $A\to \Aut(Z)$ correspond to morphisms of algebraic varieties.

However, in general the group $\Aut(Z)$ is too big to be an algebraic variety, for instance for $Z=\A^n$, $n\ge 2$ (see Example~\ref{Example:Triangular}).
\end{remark}

We will observe (in Lemma~\ref{Lem:Shafa} below) that this topology, when restricted to $\Aut(\A^n_\k)$, gives the one introduced by Shafarevich in \cite{Sha1}, i.e.~the inductive limit topology given by the inclusion of affine algebraic varieties $$\Aff(\A^n_\k)=\Aut(\A^n_\k)_{\le 1}\subset\Aut(\A^n_\k)_{\le 2}\subset \Aut(\A^n_\k)_{\le 3}\subset \dots$$
and corresponds in fact to an infinite-dimensional algebraic group (or ind-group). In order to do this, we recall how one obtains natural structures of affine varieties of 
$\Aut(\A^n_\k)_{\le d}$ and 
$\SAut(\A^n_\k)_{\le d}$.
\begin{lemma}\label{Lemm:StructureAffine}
Let us fix some integers $d,n\ge 1$, and see $\End(\A^n_\k)_{\le d}$ as an affine space, via the bijection $\End(\A^n_\k)\to (\k[x_1,\dots,x_n]_{\le d})^n$. Then, the following hold:
\begin{enumerate}[$(1)$]
\item
$J_{\k^*}=\{ f\in \End(\A^n_\k)_{\le d}\mid \Jac(f)\in \k^*\}$ is locally closed is $\End(\A^n_\k)_{\le d}$, and inherits from it the structure of an affine variety.
\item
$J_{1}=\{ f\in \End(\A^n_\k)_{\le d}\mid \Jac(f)=1\}$ is closed in $\End(\A^n_\k)_{\le d}$.
\item
$\Aut(\A^n_\k)_{\le d}$ is a closed subset of 
$J_{\k^*}$.
\item
$\SAut(\A^n_\k)_{\le d}$ is a closed subset of 
$J_1$.
\end{enumerate}
\end{lemma}
\begin{proof}The Jacobian being a morphism $\End(\A^n_\k)_{\le d}\to \k[x_1,\dots,x_n]_{\le n(d-1)}$, the sets 
$$J_k=\{ f\in \End(\A^n_\k)_{\le d}\mid \Jac(f)\in \k\}\mbox{ and }J_1$$ are closed in $\End(\A^n_\k)_{\le d}$ and are thus affine algebraic varieties. Since $J_0=\{ f\in \End(\A^n_\k)_{\le d}\mid \Jac(f)=0\}$ is given by one equation in $J_\k$, the set $J_{\k^*}=J_\k\setminus J_0$ is affine (and locally closed in $\End(\A^n_\k)_{\le d}$). This yields assertions $(1)$ and $(2)$.

In order to show $(3)$ and $(4)$, we define $W_d$ to be the set of non-zero  $(n+1)$-uples $(h_0,\dots,h_n)$
of homogeneous polynomials $h_i\in \k[x_0,\dots,x_n]$ of degree $d$, where $h_0=\mu x_0^d$, $\mu\in \k$,
up to linear equivalence: $(h_0,\dots,h_n)\sim (\lambda h_0,\dots,\lambda h_n)$ for any $\lambda\in \k^{*}$.
The equivalence class of $(h_0,\dots,h_n)$ will be denoted by $[h_0:\dots:h_n]$. 
The set of homogeneous polynomials of degree $d$ in $n+1$ variables being a $\k$-vector space, this gives to $W_d$ a canonical structure of projective space.  We then denote by $B_d\subset W_d$ the hyperplane given by $h_0=0$ and obtain a canonical isomorphism of affine spaces
$$\begin{array}{ccc}
W_d\setminus B_d&\stackrel{\sim}{\longrightarrow}& \End(\A^n_\k)_{\le d}\\
\ [x_0^d:h_1:\dots:h_n]&\mapsto &(h_1(1,x_1,\dots,x_n),\dots,h_n(1,x_1,\dots,x_n)).\end{array}$$

We denote by $Y\subseteq  W_{d^{n-1}} \times (W_d\setminus B_d)$ the set consisting of elements $(g,f)$, such that $h:=(g_0(f_0,\ldots,f_n) , \ldots , g_n(f_0,\ldots,f_n) )$ is a multiple (maybe $0$) of the identity, i.e. $h_ix_j=h_jx_i$ for all $i,j$.  
The description of $Y$ shows that it is closed in $ W_{d^{n-1}} \times (W_d\setminus B_d) $. Since $W_{d^{n-1}}$ is a complete variety, the projection $p_2 \colon W_{d^{n-1}} \times (W_d\setminus B_d) \to (W_d\setminus B_d)$ is a Zariski-closed morphism, so $p_2(Y)$ is closed in $(W_d\setminus B_d)\simeq \End(\A^n_\k)_{\le d}$.

In order to show $(3)$ and $(4)$, we only need to show that 
$p_2(Y)\cap J_{\k^*}=\Aut(\A^n_\k)_{\le d}$, which implies that $p_2(Y)\cap J_{1}=\SAut(\A^n_\k)_{\le d}$. We then show both inclusions.

$(i)$ If $f\in \Aut(\A^n_\k)_{\le d}$, there exists $g\in \Aut(\A^n_\k)_{\le d^{n-1}}$ such that $g\circ f=\mathrm{id}$ (\cite[Theorem 1.5, page 292]{BCW}). In consequence, we obtain $(g,f)\in Y$, so $f\in p_2(Y)$. The fact that $f$ belongs to $J_{\k^*}$ is given by Remark~\ref{Rem:ConjJac}, so $\Aut(\A^n_\k)_{\le d}\subseteq p_2(Y)\cap J_{\k^*}$.

$(ii)$ Let $(g,f)\in Y$, with $f\in J_{\k^*}$. By definition of $Y$, the element $(g_0(f_0,\ldots,f_n)$, $\ldots$, $g_n(f_0,\ldots,f_n) )$ is a multiple of the identity. There exists some $j$ such that $g_j\not=0$. The fact that $f\in J_{\k^*}$ implies that $f\colon \A^n_\k\to \A^n_\k$ is dominant. Hence, $g_j(f_0(1,x_0,\dots,x_n)$, $\dots$, $f_n(1,x_0,\dots,x_n))$ is not equal to zero.  This implies that $g\in \End(\A^{n-1})_{\le d^{n-1}}$ and that $g\circ f=\mathrm{id}$. Hence, $f\in \Aut(\A^n_\k)_{\le d}$. This yields $p_2(Y)\cap J_{\k^*}\subseteq \Aut(\A^n_\k)_{\le d}$.
\end{proof}
\begin{lemma}\label{Lem:Shafa}
Let us fix some integers $d,n\ge 1$, and see $\Aut(\A^n_\k)_{\le d}$ and $\SAut(\A^n_\k)_{\le d}$ as affine varieties, via their inclusion in $\End(\A^n_\k)_{\le d}\simeq (\k[x_1,\dots,x_n]_{\le d})^n$ $($see Lemma~$\ref{Lemm:StructureAffine})$. Then, the following hold:
\begin{enumerate}[$(1)$] 
\item
Morphisms $A\to \Aut(\A^n_\k)$ correspond to elements of $\Aut(\A^n_{\k[A]})$.
\item
For each $\k$-algebraic variety $A$, the morphisms $A\to \Aut(\A^n_\k)$ that have image in $\Aut(\A^n_\k)_{\le d}$ correspond to morphisms of algebraic varieties $A\to \Aut(\A^n_\k)_{\le d}$.
\item
The set $\Aut(\A^n_\k)_{\le d}$ is closed in $\Aut(\A^n_\k)$, and the restriction of the topology of $\Aut(\A^n_\k)$ on it yields the topology of its algebraic variety structure.
\item
A subset of $\Aut(\A^n_\k)$ is closed if and only if its intersection with $\Aut(\A^n_\k)_{\le d}$ is closed, for each integer $d\ge 1$.
\end{enumerate}
Moreover, everything works the same replacing $\Aut$ with $\SAut$.
\end{lemma}
\begin{proof} We do the proof with $\Aut$, the same proof works replacing $\Aut$ with $\SAut$.

The map $$\begin{array}{ccc}
\End(\A^n_\k)_{\le d}\times \A^n_\k&\to& \End(\A^n_\k)_{\le d}\times \A^n_\k\\
(f,x)&\mapsto &(f,f(x))\end{array}$$
being a morphism of algebraic varieties, every morphism of algebraic varieties $A\to \Aut(\A^n_\k)_{\le d}$ yields an $A$-automorphism of $A\times \A^n_\k$, and thus a morphism $A\to \Aut(\A^n_\k)$ with image in $\Aut(\A^n_\k)_{\le d}$. 

Conversely, let $f\colon A\to \Aut(\A^n_\k)$ be a morphism. By definition, this is given by an $A$-automorphism of $A\times \A^n_\k$. Composing with the projection $A\times \A^n_\k\to \A^n_\k$ and then with the function~$x_i$ on~$\A^n_\k$, we obtain an element $f_i\in \k[A\times \A^n_\k]=\k[A][x_1,\dots,x_n]$. Hence, $(f_1,\dots,f_n)$ is an element of $\Aut(\A^n_{\k[A]})$. Such an element yields a morphism of algebraic varieties $A\to \End(\A^n_\k)_{\le m}$ for some $m\in\mathbb{N}$, having image in $\Aut(\A^n_\k)$.  This yields $(1)$ and $(2)$. Assertion $(3)$ also follows, after observing that the preimage of $\Aut(\A^n_\k)_{\le i}$ is a closed subset of $A$, for each $i\ge 0$.

It remains to show $(4)$. Let $Y\subset \Aut(\A^n_\k)$ be any subset. If $Y$ is closed in $\Aut(\A^n_\k)$, then $Y\cap \Aut(\A^n_\k)_{\le d}$ is closed (in $\Aut(\A^n_\k)$ or $\Aut(\A^n_\k)_{\le d}$) for each $d$, by $(3)$. Suppose conversely that $Y\cap \Aut(\A^n_\k)_{\le d}$ is closed for each $d$. In order to show that $Y$ is closed in $\Aut(\A^n_\k)$, we take any morphism $A\to \Aut(\A^n_\k)$, and show that the preimage of $Y$ is closed. As we observed before, we can see this morphism as a morphism of algebraic varieties $A\to \Aut(\A^n_\k)_{\le m}$ for some $m$. Since $Y\cap \Aut(\A^n_\k)_{\le m}$ is closed, the preimage of $Y$ is closed.
\end{proof}
\subsection{Dynamical properties and the behaviour at infinity}\label{SubSec:Dynamical}
In the sequel, we fix a canonical open embedding $\A^n_\k\to \p^n_\k$ given by $$(x_1,\dots,x_n)\mapsto [1:x_1:\dots:x_n]$$
and denote by $H_\infty\subset \p^n_\k$ the complement $H_\infty=\p^n_\k\setminus \A^n_\k$, which is a hyperplane.
Once this is fixed, we have a canonical extension of any element of $\End(\A^n_\k)$ to a rational map $\p^n_\k\dasharrow \p^n_\k$. In particular, this yields inclusions $$\Aut(\A^n_\k)\longrightarrow \Bir(\A^n_\k)\stackrel{\simeq}{\longrightarrow} \Bir(\p^n_\k).$$ 
The automorphisms of degree $1$ (affine automorphisms) correspond to those which extend to elements of $\Aut(\p^n_\k)$. The others extend to birational maps which are not automorphisms, and we can associate to them two sets, which are classical objects in dynamics (see for example \cite{Si,GuSi,Bisi}):
\begin{definition}
For each $f\in \Aut(\A^n_\k)$ of degree $\ge 2$, we define $I_f\subset \p^n_\k$ to be the indeterminacy locus of the extension of $f$ to $\p^n_\k$, and $X_f\subset \p^n_\k$ to be the image of $H_\infty\setminus (I_f)$.
\end{definition}
In order to see the sets $I_f,X_f$ explicitely, we use the following definition.
\begin{definition}
If $f=(f_1,\dots,f_n)\in \Aut(\A^n_\k)$ is an element of degree~$d$ and $a_i$ is the homogeneous part of $f_i$ of degree~$d$, for $i=1,\dots,n$, we say that $(a_1,\dots,a_n)\in \End(\A^n_\k)$ is the \emph{highest homogeneous part of $f$}.
\end{definition}
\begin{remark}
If $(a_1,\dots,a_n)$ is the highest homogeneous part of $f\in \Aut(\A^n_\k)$, the set $I_f\subset \p^n_\k$ is given by
$$x_0=0, a_i=\dots=a_n=0$$
and is a proper closed subset of the hyperplane at infinity $H_\infty$ (given by $x_0=0$).

Moreover, the set $X_f$ is equal to 
$$\{[0:a_1(x_1,\dots,x_n):\dots:a_n(x_1,\dots,x_n)]\mid [0:x_1:\dots:x_n]\in H_\infty\setminus I_f\}.$$
 \end{remark}
 \begin{remark}[The regular terminology]In dynamics, the elements $f\in \Aut(\A^n_\k)$ such that $I_f\cap I_{f^{-1}}=\emptyset$ are called ``regular'' and elements such that $X_f\cap I_f=\emptyset$ are called ``weakly-regular'' in \cite{GuSi} or ``quasi-regular'' in \cite{Bisi}. 
 The use of this terminology in algebraic geometry can be confusing, because of the common use of the word regular for maps (in fact all elements of $\Aut(\A^n_\k)$ are  morphisms, hence biregular). This is why we will say ``dynamically regular'' if $I_f\cap I_{f^{-1}}=\emptyset$; will will not use the words ``quasi-regular'' or ``weakly-regular''.\end{remark}
 \begin{remark}
 Fixing the degree $d$, we obtain an algebraic variety $\Aut(\A^n_\k)_d$. It follows from the definition that the set of dynamically regular elements of $\Aut(\A^n_\k)_d$ is an open subset (in general not dense). However, the set of dynamically regular elements of $\Aut(\A^n_\k)_{\le d}$ (and thus of $\Aut(\A^n_\k)$) is not open in general. For $n=2$, this can be seen by taking for example $$f=(-x_2,x_1+(x_2)^2)\in \SAut(\A^2_\k), \alpha=(x_1+t(x_2)^2,x_2)\in \SAut(\A^2_{\k[t]})$$ and considering the family $\beta=\alpha f \alpha^{-1}\in \SAut(\A^2_{\k[t]})$. Then, $\beta(t)$ is not  dynamically regular for $t\in \k^*$ but $\beta(0)=f$ is dynamically regular.
\end{remark}
\begin{remark}\label{Remark:Dimension2Xf}
If $n=2$, we find $X_f=I_{f^{-1}}$. Indeed, the extension of an element $f\in \Aut(\A^2_\k)$ of degree $>1$ is an element of $\Bir(\p^2_\k)$ which contracts the line at infinity onto one point $X_f$. The inverse of this birational map is then defined at any other point of $H_\infty$, so we find $X_f=I_{f^{-1}}$.
\end{remark}
 
We now show that the degree of a composition is determined by the sets $I_f$ and $X_f$ defined above. The following results are classical in the world of dynamics, we recall the easy proofs for self-containedness.
\begin{lemma}\label{Lem:ConditionDegree}
Let $f,g\in \Aut(\A^n_\k)$ be of degree $\ge 2$. Then, the following are equivalent:
\begin{enumerate}[$(1)$]
\item $\deg(g f)=\deg(g)\cdot \deg(f),$
\item The set $X_f$ is not contained in $I_{g}$.
\end{enumerate}
Moreover, if both conditions hold, then $X_{gf}\subset X_g$.
\end{lemma}
\begin{proof}
Denoting by $(a_1,\dots,a_n)$ and $(b_1,\dots,b_n)$ the highest homogeneous parts of $f$ and $g$ respectively, the equality $\deg(g f)=\deg(g)\cdot \deg(f)$ is equivalent to the fact that one of the polynomials 
$$b_1(a_1,\dots,a_n),\dots,b_n(a_1,\dots,a_n)$$
is not equal to zero.

By definition, the sets $I_g,X_f\subset H_\infty$ are given respectively by 
$$\begin{array}{rcl}
I_g&=&\{[0:y_1:\dots:y_n]\in H_\infty\mid b_1(y_1,\dots,y_n)=\dots=b_n(y_1,\dots,y_n)\},\\
X_f&=&\{[0:a_1(y_1,\dots,y_n):\dots:a_n(y_1,\dots,y_n)]\mid [0:y_1:\dots:y_n]\in H_\infty\setminus I_f\}.\end{array}$$

If $X_f\not\subset I_g$, there exists a point $[0:y_1:\dots:y_n]\in H_\infty\setminus I_f$ such that $[0:a_1(y_1,\dots,y_n):\dots:a_n(y_1,\dots,y_n)]\notin I_g$, which corresponds to the existence of an index $i\in \{1,\dots,n\}$ such that $b_i(a_1(y_1,\dots,y_n),\dots,a_n(y_1,\dots,y_n))\not=0$. In particular, the polynomial $b_i(a_1,\dots,a_n)$ is not zero, so $\deg(gf)=\deg(g)\deg(f)$.

Conversely, if $b_i(a_1,\dots,a_n)$ is a non zero polynomial, the open subset $U_i\subset H_\infty$ corresponding to the non-vanishing of this polynomial is non-empty. Intersecting this open set with $H_\infty\setminus I_f$ yields a non-empty open subset of points $H_\infty$ which have image in $X_f$ and not in $I_g$.

Now that the equivalence between $(1)$ and $(2)$ is shown, we show that these imply that $X_{gf}\subset X_{g}$. Since $\deg(gf)=\deg(g)\deg(f)$, the homogeneous part of higher degree of $gf$ is $(b_1(a_1,\dots,a_n),\dots,b_n(a_1,\dots,a_n)).$ This implies that points of $X_{gf}$ are image by the extension of $g$ of points of $X_f\setminus I_g$, and then lie in $X_g$.
\end{proof}
\begin{corollary}\label{Coro1}
Let $f,g\in \Aut(\A^n_\k)$ be of degree $\ge 2$. Then, the following hold: 
\begin{enumerate}[$(1)$]
\item
$X_f\subset I_{f^{-1}}.$
\item
If $I_g\cap I_{f^{-1}}=\emptyset$, then $\deg(gf)=\deg(g)\cdot \deg(f).$
\item
$X_f\not\subset I_f\Leftrightarrow \deg(f^2)=\deg(f)^2.$
\end{enumerate}
\end{corollary}
\begin{proof}
Part $(1)$ follows from Lemma~\ref{Lem:ConditionDegree} and the fact that $\deg(f^{-1}f)<\deg(f^{-1})\deg(f)$. If $I_g\cap I_{f^{-1}}=\emptyset$, then $X_f\not\subset I_g$ by $(1)$, so the equality $\deg(gf)=\deg(g)\cdot \deg(f)$ follows again from Lemma~\ref{Lem:ConditionDegree}. Part $(3)$ corresponds to Lemma~\ref{Lem:ConditionDegree}, in the case $f=g$.
\end{proof}
\begin{corollary}\label{Cor:WeaklyregSequence}
If $f\in \Aut(\A^n_\k)$ is an element such that $X_f\cap I_f=\emptyset$, then $$\deg(f^m)=\deg(f)^m\mbox{ and }X_{f^m}\subset X_f$$ for each $m\ge 1$.
\end{corollary}
\begin{proof}
We prove the result by induction on $m$, the case $m=1$ being obvious. For $m\ge 2$, we use the facts that $X_f\cap I_f=\emptyset$ and $X_{f^{m-1}}\subset X_f$, which imply that $X_{f^{m-1}}\not\subset I_f$. Applying Lemma~\ref{Lem:ConditionDegree} to the composition $f\circ f^{m-1}$, we obtain that $\deg(f^m)=\deg(f^{m-1})\deg(f)$, which is equal to $\deg(f)^m$ by induction hypothesis, and also that $X_{f^m}\subset X_f$.\end{proof}
Restricting to dimension $2$, we obtain the following result. 
\begin{corollary}\label{Coro:Dim2}
Let $f\in \Aut(\A^2_\k)$ of degree $\ge 1$. The following are equivalent:
\begin{enumerate}[$(1)$]
\item
$I_f\not= I_{f^{-1}}$;
\item
$I_f\cap  I_{f^{-1}}=\emptyset$;
\item
$\deg(f^2)=\deg(f)^2$;
\item
$\deg(f^m)=\deg(f)^m$ for each $m\ge 1$.
\end{enumerate}
\end{corollary}
\begin{proof}
As we are in dimension $2$, we have $X_f=I_{f^{-1}}$ and $X_{f^{-1}}=I_f$, which are two points of $H_\infty$, which can be distinct or not (see Remark~\ref{Remark:Dimension2Xf}).

The equality $\deg(f^2)=\deg(f)^2$ is equivalent to $X_f\not\subset I_f$ (Corollary~\ref{Coro1}), which is here equivalent to $I_{f^{-1}}\not= I_f$ or $I_{f^{-1}}\cap I_f=\emptyset$. Hence, $(1),(2),(3)$ are equivalent, and of course implied by $(4)$. It remains to see that $(2)$ corresponds to $X_f\cap I_f=\emptyset$, which implies that $\deg(f^m)=\deg(f)^m$ for each $m\ge 1$ by Corollary~\ref{Cor:WeaklyregSequence}.
\end{proof}
\begin{remark}
The most interesting implication of this corollary is $(3)\Rightarrow (4)$, i.e.~that $\deg(f^2)=\deg(f)^2$ implies that $\deg(f^m)=\deg(f)^m$ for $m\ge 1$, a fact arleady observed by Jean-Philippe Furter in \cite{FurterIt} (at least when $\car(\k)=0$).
\end{remark}
Looking at the proof, this result has no reason to be true in dimension $n\ge 3$, and is in fact false on the easiest non-trivial example, that we describe now.
\begin{example}
Let $f=(x_1+(x_2)^2,x_2+(x_3)^2,x_3)\in \SAut(\A^3_\k)$, which has highest homogeneous part $a=((x_2)^2,(x_3)^2,0)\in \End(\A^3_\k)$. Then, 
$$I_f=[0:1:0:0],\  X_f=\{[0:x_1:x_2:0]\mid [x_1:x_2]\in \p^1\}.$$
Since $X_f\not\subset I_f$, we have $\deg(f^2)=\deg(f)^2=4$ and $X_{f^2}\subset X_f$. More precisely, the homogeneous part of $f^2$ is $a^2=((x_3)^4,0,0)$, so 
$$I_{f^2}=I_f=X_{f^2}=[0:1:0:0].$$ In particular, $X_{f^2}\subset I_f$, so $\deg(f^3)<\deg(f^2)\cdot \deg(f)$.

In fact, one easily checks with the formulas that $\deg(f^n)\le 4$ for each $n$, and that $\deg(f^n)=4$ for each $n\ge 2$  if $\car(\k)=0$, indeed
$$\begin{array}{c}f^n=\left(x_1+n(x_2)^2+n(n-1)x_2(x_3)^2+(\sum_{i=1}^{n-1} i^2)(x_3)^4,x_2+n(x_3)^2,x_3\right)\end{array}$$
for each $n\ge 0$.
\end{example}

\subsection{Families of automorphisms and valuations}\label{SubSec:Families}
In the sequel, we will study families of elements of $\Aut(\A^n_\k)$, which correspond to elements of $\Aut(\A^n_{\k((t))})$.

It is then natural to use the valuation $$\nu\colon \k((t))[x_1,\dots,x_n]\to \Z\cup \{-\infty\}$$ associated to $t$. We define precisely this valuation here, as we will use it often afterwards.
\begin{definition}
Every element  $f\in\k((t))[x_1,\dots,x_n]\setminus \{0\}$ can be written as
$$f=\sum_{k=m}^{\infty} a_k t^k$$
where $m\in \mathbb{Z}$, $a_i\in \k[x_1,\dots,x_n]$ for $i\ge m$, and $a_m\not=0$. We then define
$\nu(f)=m$. 
Choosing $\nu(0)=-\infty$, we obtain a valuation 
$$\nu\colon \k((t))[x_1,\dots,x_n]\to \Z\cup \{-\infty\}.$$
\begin{enumerate}
\item
If $\nu(f)=0$, we define $f(0)=a_0$.
\item
If $\nu(f)>0$, we define $f(0)=0$.
\item
If $\nu(f)<0$, we say that \emph{$f$ has a pole at $t=0$} and that $f(0)$ is not defined.
\end{enumerate}
\end{definition}
\begin{definition}\label{Defi:DefinedOriginAut}
Let $f=(f_1,\dots,f_n)\in\Aut(\A^n_{\k((t))})$. We define $$\nu(f)=\min\{\nu(f_i)\mid i=1,\dots,n\}.$$
\begin{enumerate}
\item
If $\nu(f)\ge 0$, we say that $f$ is \emph{defined at the origin} (or \emph{has a value}), and define $$f(0)=(f_1(0),\dots,f_n(0)),$$ to be its value, which is an element of $\End(\A^n_\k)$.
\item
 If $\nu(f)<0$, we say that \emph{$f$ has a pole at $t=0$}, and say that $f(0)$ is not defined. 
\end{enumerate}
\end{definition}
\begin{remark}
In the above definition, it is possible that the element $f(0)$ is defined, but does not belong to $\Aut(\A^n_\k)$. This is for example the case when one of the components becomes $0$, or for  $f=(tx_1+x_2,x_2,\dots,x_n)$.
This phenomenon is however impossible for $\SAut(\A^n_\k)$, as the following result shows.
\end{remark}

\begin{lemma}\label{Lem:FormalInverse}
Let $\alpha\in \SAut(\A^n_{\k((t))})$ be an element which has no pole at $t=0$. 

Then, $\alpha^{-1}$ has no pole at $t=0$, and replacing $t$ with $0$ yields two automorphisms $$\beta=\alpha(0),\gamma=\alpha^{-1}(0)\in \SAut(\A^n_{\k}).$$
such that $\beta\gamma=\mathrm{id}$.
\end{lemma}
\begin{proof}
Since $\alpha$ has no pole at $t=0$, it sends $(0,\dots,0)\in \A^n_{\k((t))}$ onto an element having coordinates in $\k[[t]]$. We can thus replace $\alpha$ by its composition with a translation and assume that $\alpha$ (and thus $\alpha^{-1}$) fixes the origin. Its linear part is then equal to an element of $\SL(n,\k[[t]])$, whose inverse also belongs to $\SL(n,\k[[t]])$. Replacing with the composition by this inverse, we can assume that $\alpha$ has a trivial linear part. We denote by $\m$ the ideal of $\k((t))[x_1,\dots,x_n]$ generated by the $x_i$, and can then write
$$\alpha=(f_1,\dots,f_n)$$
for some $f_1,\dots,f_n\in \k[[t]][x_1,\dots,x_n]$, $f_i\equiv x_i\pmod {\m^2}$. We write then
$$
\alpha^{-1}=(x_1+\sum_{i=2}^d g_{i,1},\dots,x_n+\sum_{i=2}^{d_2} g_{i,n})$$
where the $ g_{i,j}\in \k((t))[x_1,\dots,x_n]$ are homogeneous polynomials of degree $i$ and $d_2$ is the degree of $g$. Assume for contradiction that one of the $g_{k,j}$ does not belong to $\k[[t]][x_1,\dots,x_n]$, and choose $k$ to be minimal for this. Then, the $j$-th coordinate of $\alpha^{-1}\circ \alpha$ is equal to
$$x_j=f_j+\sum_{i=2}^{d_2} g_{i,j}(f_1,\dots,f_n),$$
which implies that $\sum_{i=k}^d g_{i,j}(f_1,\dots,f_n)\in \k[[t]][x_1,\dots,x_n].$ But the  part of degree $k$ of this sum is in fact equal to $g_{i,j}(x_1,\dots,x_n)$, which does not belong to $\k[[t]][x_1,\dots,x_n]$.

We have proved that $\alpha^{-1}$ does not have any pole at the origin. We then denote by $d_1,d_2$ the degrees of $\alpha$ and $\alpha^{-1}$ (which are the maximal degree of their components), and denote by $\mathrm{End}_{d_i}$ the set of endomorphisms of $\A^n$ of degree $\le d_i$, which is naturally isomorphic to an affine space. Observe that  $\alpha,\alpha^{-1}$ corresponds to a $\k((t))$-point of the algebraic variety 
$$\{(f,g)\in \mathrm{End}_{d_1}\times \mathrm{End}_{d_1}\mid f\circ g=\mathrm{id}\}.$$
Since neither $\alpha$ neither $\alpha^{-1}$ has a pole at $t=0$, all coefficients are defined at the origin. Replacing $t$ with $0$ gives then the result.
\end{proof}
\begin{remark}
The result of Lemma~\ref{Lem:FormalInverse} can also be obtained by the fact that each $\SAut(\A^n_\k)_{\le d}$ is closed in $\End(\A^n_\k)_{\le d}$ (Lemma~\ref{Lemm:StructureAffine}).
\end{remark}
We will apply the following classical valuative result, and recall the argument of the proof, given in \cite[$\S 2.1$]{Furter} (the version that we need here is slightly more general, but the proof is analogue).
\begin{lemma}\label{Lemm:Valuative}
Let $Y,Z$ be two quasi-projective $\k$-algebraic varieties, let $\varphi\colon Y\to Z$ be a morphism and let $z\in Z$ be a $($closed$)$ point of $Z$. The following assertions are equivalent:
\begin{enumerate}[$(1)$]
\item
The point $z$ belongs to the closure $\overline{\varphi(Y)}$ of the image.
\item
There is an irreducible $\k$-curve $\Gamma$, a smooth closed point $p\in \Gamma$ and a rational map $\iota\colon \Gamma\dasharrow Y$ such that $\varphi\circ \iota\colon \Gamma\dasharrow Z$ is defined at $p$ and sends it onto $z$.
\item
There is a $\k((t))$-point $y\in Y(\k((t)))$ such that $\varphi(y)\in Z(\k((t)))$ has no pole at $t=0$ and $\varphi(y)(0)=z$.
\end{enumerate}
\end{lemma}
\begin{proof}
$(1)\Rightarrow (2)$: We can see $Y$ and $Z$ as dense open subset of projective algebraic varieties~$\overline{Y},\overline{Z}$, and assume that $\varphi$ extends to a morphism $\overline{\varphi}\colon \overline{Y}\to \overline{Z}$, whose image is then closed. As $z\in Z\subset \overline{Z}$ belongs to the closure of $\varphi(Y)$, it is the image of a point $p\in\overline{Y}$. Taking a closed irreducible curve $\Gamma\subset \overline{Y}$ intersecting $Y$ and passing through $p$ (which exists since $Y$ is dense in $\overline{Y}$), we obtain a morphism $\Gamma\to \overline{Z}$. It remains to choose $\iota\colon \Gamma\dasharrow Y$ to be the rational map given by the inclusion $\Gamma\cap Y\to Y$.

$(2)\Rightarrow (3)$ The rational maps $\Gamma\dasharrow Y\to Z$ correspond to field homomorphisms $\k(Z)\to \k(Y)\to \k(\Gamma)$, sending the local ring $\mathcal{O}_{z,Z}$ to $\mathcal{O}_{p,\Gamma}$.

Denote by $\widehat{\mathcal{O}_{p,\Gamma}}$ the completion of $\mathcal{O}_{p,\Gamma}$, with respect to its maximal ideal.
Because $\mathcal{O}_{p,\Gamma}$ is a Noetherian regular local ring of dimension $1$ with residue field $\k$, its completion is a complete Noetherian regular local ring with the same properties and by Cohen Theorem, it must be isomorphic to a ring of formal power series. The dimension being~$1$, one has a $\k$-isomorphism $\widehat{\mathcal{O}_{p,\Gamma}}\simeq \k[[t]]$, which induces a field homomorphism $\k(\Gamma)\to \k((t))$. 

The composition $\k(Y)\to \k(\Gamma)\to \k((t))$ corresponds to the $\k((t))$-point $y$ that we want, and its image corresponds to the composition $\k(Z)\to \k(Y)\to \k(\Gamma)\to \k((t))$.
\end{proof}
\begin{corollary}\label{Coro:LimitConjugateValuation}
Let $f\in \SAut(\A^n_\k)$, let $d\ge 1$ be an integer and let $Y$ be the $\k$-algebraic variety $\SAut(\A^n_\k)_{\le d}$. The following assertions are equivalent:
\begin{enumerate}[$(1)$]
\item
The set $\{gfg^{-1}\mid g\in Y\}$ is closed in $\SAut(\A^n_\k)$.
\item
If $\Gamma$ is an irreducible $\k$-curve $\Gamma$ and $\iota\colon \Gamma\dasharrow Y$ is a rational map such that $\varphi\circ \iota\colon \Gamma\dasharrow Z$ is defined at a smooth point $p\in \Gamma$, the image of $p$ belongs to $\{gfg^{-1}\mid g\in Y\}$.
\item
If $\varphi\in Y(\k((t)))$ is an element that has poles at $t=0$ and $h=\varphi f \varphi^{-1}$ has no poles at $t=0$, then $h(0)$ belongs to $\{gfg^{-1}\mid g\in Y\}$.
\end{enumerate}
\end{corollary}
\begin{proof}
The degree of the inverse of an element $g\in Y$ is at most $d^{n-1}$ (see \cite[Theorem 1.5, page 292]{BCW}). Hence, the map $Y\to \SAut(\A^n_\k)$ that sends $g$ onto $gfg^{-1}$ corresponds to a morphism of algebraic varieties $\varphi\colon Y\to \SAut(\A^n_\k)_{\le m}$ for some $m$.
The result follows then from Lemma~\ref{Lemm:Valuative} applied to $\varphi$, and Lemma~\ref{Lem:Shafa}.
\end{proof}
\section{Conjugacy classes of dynamically regular automorphisms of $\A^n_\k$}\label{Sec:ConjRegular}
\subsection{Image at infinity of elements of $\Aut(\A^n_{\k((t))})$}
In $\S\ref{SubSec:Dynamical}$, we explained how the set $X_f\subset H_\infty$ is defined, for each element $f\in \Aut(\A^n_{\k})$, by extending the map to $\p^n_\k$ and looking at the image of the hyperplane $H_\infty$ at infinity.

We now associate similarly a subset $X_\alpha\subset H_\infty$ to an element $\alpha\in\Aut(\A^n_{\k((t))})$ which has a pole at $t=0$, by taking the limit of the image of $\alpha(t)$, when $t$ goes towards $0$. The formal definition of $X_\alpha$ is the following:
\begin{definition}
Let $\alpha\in \Aut(\A^n_{\k((t))})$ be an element of valuation $\nu(\alpha)=-m<0$, that we write 
 $$\alpha=\left(\frac{1}{t^m}\alpha_1,\dots,\frac{1}{t^m}\alpha_n\right)$$
 where $\alpha_1,\dots,\alpha_n\in \k[[t]][x_1,\dots,x_n]$ are such that $$\tilde{\alpha}=(\alpha_1(0),\dots,\alpha_n(0))=(\tilde{\alpha}_1,\dots,\tilde{\alpha}_n)\in \End(\A^n_\k)\setminus \{0\}.$$ We then define $$X_\alpha=\{[0:\tilde{\alpha}_1(y_1,\dots,y_n):\dots:\tilde{\alpha}_n(y_1,\dots,y_n)]\mid (y_1,\dots,y_n)\in \A^n_\k\setminus \tilde{\alpha}^{-1}(\{0\}) \}\subset H_\infty.$$
\end{definition}
This definition can be geometrically understood:
\begin{remark}In the above definition, $\alpha$ is not defined at $t=0$, but extending $\alpha$ to $\p^n$ we obtain the element of $\Bir(\p^n_{\k((t))})$ given by
 $$[x_0:\dots:x_n]\dasharrow[t^mx_0^d:F_1(x_0,\dots,x_n,t):\dots:F_n(x_0,\dots,x_n,t)],$$
 where $d$ is the degree of $\alpha$ and each $F_i(x_0,\dots,x_n,t)\in \k[[t]][x_0,\dots,x_n]$ is the homogeneisation of $\alpha_i$. This corresponds to a family of rational maps of $\p^n_\k$ parametrised by~$t$, which as a value at $t=0$, corresponding to 
 $$[x_0:\dots:x_n]\dasharrow[0:F_1(x_0,\dots,x_n,0):\dots:F_n(x_0,\dots,x_n,0)].$$
 The set $X_\alpha\subset H_\infty$ is then the image of this map by points of $\A^n_\k$ which are well-defined under this map.
 
 Writing $[1:x_1:\dots:x_n]$ such a point, the image of it by the extension of $\alpha$ is 
 $$[t^m:F_1(1,\dots,x_n,t):\dots:F_n(1,\dots,x_n,t)]=[t^m:\alpha_1(x_1,\dots,x_n,t):\dots:\alpha_n(x_1,\dots,x_n,t)]$$
 and corresponds to a curve in $\p^n$, whose point when $t=0$ belongs to $X_\alpha$.
 
 The set $X_\alpha$ corresponds then to the limit, viewed in $\p^n_\k$, of the image of points of $\A^n_\k$ under $\alpha(t)$ when $t$ goes towards $0$.
\end{remark}

\begin{proposition}\label{Prop:Ifpoles}
Let $f\in \SAut(\A^n_\k)$ be an element of degree $d>1$, and let $\alpha\in \SAut(\A^n_{\k((t))})$ be an element that has a pole at $t=0$. If $$X_\alpha\not\subset I_f,$$ then $\alpha^{-1} f \alpha$ has a pole at $t=0$.
\end{proposition}
\begin{proof}
We write $$f=(f_1,\dots,f_n)$$
 for some $f_1,\dots,f_n\in \k[x_1,\dots,x_n]$ and denote by  $(a_1,\dots,a_n)$ the highest homogeneous part of $f$ (which is of degree $d$). In particular, the indetermination locus $I_f\subset \p^n_\k$ of the extension of $f$ to $\p^n_\k$ is given by 
 $$x_0=0, a_1=\dots=a_n=0.$$
 Because $\alpha$ has a pole at $t=0$, we have $\nu(\alpha)=-m<0$ and can write
 $\alpha=\left(\frac{1}{t^m}\alpha_1,\dots,\frac{1}{t^m}\alpha_n\right)$
 where $\alpha_1,\dots,\alpha_n\in \k[[t]][x_1,\dots,x_n]$ are such that $\tilde{\alpha}=(\alpha_1(0),\dots,\alpha_n(0))=(\tilde{\alpha}_1,\dots,\tilde{\alpha}_n)\in \End(\A^n_\k)\setminus \{0\},$ and $X_\alpha$ is then equal to $$X_\alpha=\{[0:\tilde{\alpha}_1(y_1,\dots,y_n):\dots:\tilde{\alpha}_n(y_1,\dots,y_n)]\mid (y_1,\dots,y_n)\in \A^n_\k\setminus \tilde{\alpha}^{-1}(\{0\}) \}\subset H_\infty.$$
 
 The fact that $X_\alpha\not\subset I_f$ yields a point $(y_1,\dots,y_n)\in \A^n_\k\setminus\tilde{\alpha}^{-1}(\{0\})$ and an integer $i\in \{1,\dots,n\}$ such that 
 $$a_i(\tilde{\alpha}_1(y_1,\dots,y_n),\dots,\tilde{\alpha}_n(y_1,\dots,y_n))\not=0.$$ 
 In particular, we have 
 $$a_i(\tilde{\alpha}_1,\dots,\tilde{\alpha}_n)=a_i(\alpha_1(0),\dots,\alpha_n(0))\in \k[x_1,\dots,x_n]\setminus \{0\}.$$
  
 This implies that $a_i(\frac{1}{t^m}\alpha_1,\dots,\frac{1}{t^m}\alpha_n)=\frac{1}{t^{md}}a_i(\alpha_1,\dots,\alpha_n)$ has valuation $-md$. Hence, $f \alpha$ has also valuation $-md$. 
 
 It remains to show that this implies that $\beta=\alpha^{-1}f\alpha$ has a pole at $t=0$. Indeed, if $\beta$ had no pole, we would have $\nu(\alpha\beta)\ge \nu(\alpha)=-m$, which is impossible since $f\alpha=\alpha\beta$ and $\nu(f\alpha)=-md$.
\end{proof}

\begin{corollary}\label{Coro:LimitConjugatesInside}
Let $f\in \SAut(\A^n_\k)$ be a dynamically regular element.

\begin{enumerate}[$(1)$]\item
If $\alpha\in \SAut(\A^n_{\k((t))})$ is such that $g=\alpha^{-1} f \alpha$ has no pole at $t=0$,  then $\alpha$ and $\alpha^{-1}$ have no pole at $t=0$. In particular, $g(0)$ is an element of $\SAut(\A^2_{\k})$ that is conjugated to $f$.
\item
For each $d\ge 1$, the set $\{gfg^{-1}\mid g\in \SAut(\A^n_{\k})_{\le d}\}$ is closed in $\SAut(\A^n_{\k})$.
\end{enumerate}
\end{corollary}
\begin{proof}
$(1)$ We suppose that $\alpha$ has a pole at $t=0$, which is is equivalent to the fact that $\alpha^{-1}$ has a pole at $t=0$ (Lemma~\ref{Lem:FormalInverse}), and show that $\alpha^{-1} f\alpha$ has a pole at $t=0$. If $X_\alpha\not\subset I_f,$ this is given by Proposition~\ref{Prop:Ifpoles}. Otherwise, we have $X_\alpha\not\subset I_{f^{-1}},$ (because $I_f\cap I_{f^{-1}}=\emptyset$ by hypothesis) and apply the proposition to $f^{-1}$. This implies that  $\alpha f^{-1} \alpha^{-1}$ has a pole at $t=0$. Hence, $\alpha f \alpha^{-1}$ has also a pole at $t=0$ by Lemma~\ref{Lem:FormalInverse}.

$(2)$ Follows from $(1)$ and Corollary~\ref{Coro:LimitConjugateValuation}.
\end{proof}

We obtain thus the following two results:
\begin{proposition}\label{Prop:IfConstructibleThenClosed}
Let $f\in \SAut(\A^n_\k)$  be a dynamically regular element having the following property: there exists a function $\tau\colon \N\to \N$ such that
for each conjugate $g\in \SAut(\A^n_\k)$ of $f$, there exists $h\in \SAut(\A^n_\k)$ of degree $\le \tau(\deg(g))$ such that $g=hfh^{-1}$.

Then, the conjugacy class of $f$ in $\SAut(\A^n_\k)$ is closed.
\end{proposition}
\begin{proof}
Let us denote by $C\subset \SAut(\A^n_\k)$ the conjugacy class of $f$ in $\SAut(\A^n_\k)$. Note that $C$ is closed (in $\SAut(\A^n_\k)$) if and only if $$C_d=\{g\in \SAut(\A^n_\k)\mid \deg(g)\le d\}$$ is closed for each $d\in\N$.

By Corollary~\ref{Coro:LimitConjugatesInside}, the set 
$$C_d'= \{hfh^{-1}\mid h\in\SAut(\A^n_\k)_{\le d}\}$$
is closed for each $d$.

By hypothesis, we have $C_d\subset C'_{\tau(d)}$ for each $d\in \mathbb{N}$, which implies that 
$$C_d=C'_{\tau(d)}\cap \Aut(\SAut(\A^n_\k)_{\le d})$$
is closed for each $d$.
\end{proof}
The additional hypothesis of Proposition~\ref{Prop:IfConstructibleThenClosed} is fullfilled for all dynamically regular elements of $\SAut(\A^2_\k)$, so we obtain that the conjugacy classes of all dynamically regular elements of $\SAut(\A^2_\k)$ is closed:
\begin{proposition}\label{Prop:ConjclassclosedHenon}
Let $f\in \SAut(\A^2_\k)$  be  a dynamically regular element.
Then, the following hold:
\begin{enumerate}[$(1)$]
\item
If $g\in \SAut(\A^2_\k)$ is conjugate to $f$, there exists $h\in \SAut(\A^2_\k)$, such that $$g=hfh^{-1}\mbox{ and }\deg(h)^2\le \deg(g)$$ 

\item
The conjugacy class of $f$ in $\SAut(\A^2_\k)$ is closed.
\end{enumerate}
\end{proposition}
\begin{remark}
The bound of $(1)$ also exists for $\Bir(\A^2_\k)$, but is really higher \cite{BC13}.
\end{remark}
\begin{proof}
Proposition~\ref{Prop:IfConstructibleThenClosed} yields $(1)\Rightarrow(2)$, so we only need to prove $(1)$.

Recall that $I_g=X_{g^{-1}}$ consists of one point, for each $g\in \Aut(\A^2_\k)$ of degree $>1$ (Remark~\ref{Remark:Dimension2Xf}).

Let $g=hfh^{-1}$ be an element of degree $d$. Replacing $h$ with $hf^{l}$, $l\in \Z$, we can assume that $$\deg(h)\le \deg(hf^l)\mbox{ for each }l\in \Z.$$
We can  also assume that $\deg(h)\ge 2$.
This implies, since $\deg(h)<\deg(hf)\cdot \deg(f^{-1})$, that $I_f=X_{f^{-1}}= I_{hf}$ (Lemma~\ref{Lem:ConditionDegree}). Similarly, we have $\deg(h)<\deg(hf^{-1})\cdot \deg(f)$, so $I_{f^{-1}}=X_f=I_{hf^{-1}}$.

Because the two points $I_f,I_{f^{-1}}\in H_\infty$ are distinct, we have $I_h\not=I_f$ or $I_h\not=I_{f^{-1}}$. 

If $I_h\not=I_f=I_{hf}$, we have $\deg(hfh^{-1})=\deg(hf) \deg(h^{-1})\ge \deg(h)\deg(h^{-1})$.

If $I_h\not=I_{f^{-1}}=I_{hf^{-1}}$, we have $\deg(hf^{-1}h^{-1})=\deg(hf^{-1}) \deg(h^{-1})\ge \deg(h)\deg(h^{-1})$.

The degree of an element of $\Aut(\A^2_\k)$ and its inverse being the same, we find $\deg(g)\ge \deg(h)^2$.\end{proof}

\begin{proof}[Proof of Theorem~$\ref{Thm:Regular}$]
Part $(1)$ and $(2)$ correspond to the statements of Corollary~\ref{Coro:LimitConjugatesInside}. Part $(3)$ is provided by Proposition~\ref{Prop:ConjclassclosedHenon}. 

It remains to show $(4)$. For each $d\in \N$, the set $$C_d=\{gfg^{-1}\mid g\in \SAut(\A^n_\k)_{\le d}\}$$ is closed in $\SAut(\A^n_\k)$ for each $d$, and the conjugacy class of $f$ is the infinite union $C=\bigcup_d C_d$. Let $A$ be an algebraic variety, $F\colon A\to \SAut(\A^n_{\k})$ be a morphism, and let $B\subset A$ be a locally closed subset which is contained in $F^{-1}(C)$. Writing $B_d=F^{-1}(C_d)$, the set $B_d$ is closed in $B$ for each $d$ and $B=\bigcup_d B_d$. Since $\k$ is uncountable and $B_d\subset B_{d+1}$ for each $d$, we obtain $B=B_m$ for some integer $m$. Hence, $B$ is contained in $F^{-1}(C_m)$, which is closed in $A$, so the closure of $B$ in $A$ is also contained in $F^{-1}(C_m)$, and thus in $F^{-1}(C)$.\end{proof}

\section{Conjugacy classes of elements in $\SAut(\A^2_K)$}\label{Sec:ConjclassDim2}
\subsection{Reminders on Jung - van der Kulk's Theorem and its applications}
For each field $K$, the Jung - van der Kulk's Theorem (\cite{Jun42}, \cite{VdK53}) asserts that the group $\Aut(\A^2_K)$ is generated by the groups
\[
	\begin{array}{rcl}
		\Aff(\A^2_K)&=&\left\{(ax_1+bx_2+e,cx_1+dx_2+f)\ |\ \left(\begin{array}{cc} a& b \\ c & d \end{array}\right)\in \GL(2,K), e,f\in K\right\}, \\
		\J(\A^2_K)&=&\left\{(ax_1+P(x_2),bx_2+c)\ |\ a,b\in K^{*}, c\in K, P\in K[x_2]\right\}.
	\end{array}
\]
Multiplying the decomposition of an element of $\SAut(\A^2_K)$ with homotheties we can assume that each one has determinant $1$. This implies that the group $\SAut(\A^2_K)$ is generated by the groups
\[
	\begin{array}{rcl}
		\SAff(\A^2_K)&=&\left\{(ax_1+bx_2+e,cx_1+dx_2+f)\ |\ \left(\begin{array}{cc} a& b \\ c & d \end{array}\right)\in \SL(2,K), e,f\in K\right\}, \\
		\SJ(\A^2_K)&=&\left\{(ax_1+P(x_2),a^{-1}x_2+c)\ |\ a\in K^{*}, c\in K, P\in K[x_2]\right\}.
	\end{array}
\]
The Jung - van der Kulk's Theorem implies that we have an amalgamated product structure on $\SAut(\A^2_K)$. It also yields the following classical result. We recall here the simple proof.
\begin{lemma}\label{Lem:Henon}
Every element $f\in \SAut(\A^2_K)$ is conjugated either to an element of $\SJ(\A^2_K)$ or to an element of the form 
$$f=a_m j_m  \dots a_1j_1$$
where $m\ge 1$ and each $a_i\in \SAff(\A^2_K)\setminus \SJ(\A^2_K)$ and each $j_i\in \SJ(\A^2_K)\setminus \SAff(\A^2_K)$.
\end{lemma}
\begin{proof}
We write $f$ a a product of elements of $A=\SAff(\A^2_K)$ and $J=\SJ(\A^2_K)$: 
$$f=a_{m} j_ma_{m-1} \dots a_1j_1a_0$$
where each $a_i\in A$ and each $j_i\in J$. By merging elements we can moreover assume that $j_i\not\in A$ for $i=1,\dots,m$, and that $a_i\not\in J$ for $i=1,\dots,{m-1}$.

Suppose first that $m=0$, which implies that $f=a_0\in A$, so extends to an element of $\Aut(\p^2_K)$. The action at infinity has a fixed point, and conjugating by an element of $\SL(2,K)$, we can assume that this point corresponds to $[0:1:0]$, which implies that $f$ preserves the pencil of lines through the point, so $f\in J$.

Suppose now that $m\ge 1$, which implies that $j_1\in J\setminus A$. We conjugate $f$ by $a_0$ and assume that $a_0=\mathrm{id}$. If $m=1$ and $a_m=a_{1}\in J$, then $f\in J$. If $m\ge 2$ and $a_{m}\in J$, we conjugate by $j_1$ and decrease $m$ by $1$. This reduces to the case $m\ge 1$,  $a_1=\mathrm{id}, a_{m}\not\in J$.\end{proof}
\begin{remark}\label{Rem:Henon}
Elements of the form $f=a_m j_m\dots a_1 j_1$, where $m\ge 1$ and each $a_i\in \SAff(\A^2_K)\setminus \SJ(\A^2_K)$ and each $j_i\in \SJ(\A^2_K)\setminus \SAff(\A^2_K)$ are usually called \emph{H\'enon automorphisms}.  Note that each element $j_i$ satisfies $X_{j_i}=I_{j_i}=[0:1:0]$, and that each $a_i$ extends to an automorphism of $\p^2_K$ that moves $[0:1:0]$. By Lemma~\ref{Lem:ConditionDegree}, this implies that 
\[\deg(f)=\prod_{i=1}^m\deg(j_i), I_f=[0:1:0], X_f=I_{f^{-1}}=a_m([0:1:0]).\]
In particular, $I_f\cap I_{f^{-1}}=\emptyset$, $\deg(f)>1$ and $\deg(f^m)=\deg(f)^m$ for each $m\ge 1$ (Corollary~\ref{Coro:Dim2}). Moreover, the conjugacy class of $f$ in $\SAut(\A^2_{\overline{K}})$ is closed, where $\overline{K}$ is the algebraic closure of $K$  (Proposition~\ref{Prop:ConjclassclosedHenon}).
\end{remark}
\begin{remark}\label{Rem:AlgebraicSJ}
Elements of $\SJ(\A^2_K)$ are algebraic since they are contained in an algebraic subgroup of $\Aut(\A^2_K)$ (or equivalently have a bounded degree sequence). H\'enon automorphisms are not algebraic as their degree sequence is not bounded. Hence, Lemma~\ref{Lem:Henon} corresponds to saying that non-algebraic elements are conjugate to H\'enon automorphisms, and algebraic automorphisms to elements of $\SJ(\A^2_K)$.
\end{remark}
The following corollary follows from Lemma~\ref{Lem:Henon}, and also holds (with the same proof) for $\Aut(\A^2_\k)$. This was already observed in \cite{FurterIt}.
\begin{corollary}\label{Coro:Algclosed}\item
\begin{enumerate}[$(1)$]
\item
An element $f\in \SAut(\A^2_\k)$ is algebraic if and only if $\deg(f^2)\le \deg(f)$.
\item

The subset $\SAut(\A^2_\k)_{\mathrm{alg}}$ of algebraic elements of $\SAut(\A^2_\k)$ is closed.
\end{enumerate}

\end{corollary}
\begin{proof}
$(1)(a)$ If $f$ is algebraic it is conjugate to an element $j\in\SJ(\A^2_K)$ (Lemma~\ref{Lem:Henon}, or more precisely Remark~\ref{Rem:AlgebraicSJ}).   We can thus write $f=\alpha^{-1} j\alpha$ where $\alpha\in \SAut(\A^2_K)$. We write then $\alpha=a_1j_1\dots$, where each $a_i\in \SAff(\A^2_K)\setminus \SJ(\A^2_K)$ and each $j_i\in \SJ(\A^2_K)$ (if $\alpha$ starts with an element of $\SJ(\A^2_K)\setminus \SAff(\A^2_K)$ we can merge this element with~$j$). 
If $j\notin \SAff(\A^2_K)$, then $\deg(f)=\prod \deg(j_i)^2 \deg(j)=\deg(\alpha^{-1})\deg(j)\deg(\alpha)$ (see Remark~\ref{Rem:Henon}) and thus $\deg(f^2)=\deg(\alpha^{-1} j^2 \alpha)\le \deg(\alpha^{-1})\deg(j^2)\deg(\alpha)\le \deg(f)$ since $\deg(j^2)\le \deg(j)$. If $j\in \SAff(\A^2_K)\cap \SJ(\A^2_K)$, then we can replace $j$ with $(a_1)^{-1} j a_1$ and continue like this to obtain either the previous case or $f=\alpha^{-1} a\alpha$ where $a\in \SAff(\A^2_K)$, $\alpha=j_1a_1\dots$, each $a_i\in \SAff(\A^2_K)\setminus \SJ(\A^2_K)$ and each $j_i\in \SJ(\A^2_K)\setminus \SAff(\A^2_K)$. We obtain similarly $\deg(f^2)\le \deg(f)$.

$(1)(b)$ If $f$ is not algebraic, then $f$ is conjugate to a H\'enon map $h=a_m j_m  \dots a_1j_1$, where each $a_i\in \SAff(\A^2_K)\setminus \SJ(\A^2_K)$ and each $\SJ(\A^2_K)\setminus \SAff(\A^2_K)$ (Lemma~\ref{Lem:Henon}, or more precisely Remark~\ref{Rem:AlgebraicSJ}). Writing $f=\alpha^{-1} h\alpha$ where $\alpha\in \SAut(\A^2_K)$, we can write as before $\alpha$ as a product of elements of $\SAff(\A^2_K)$ and $\SJ(\A^2_K)$, replace $h$ if $\alpha$ starts with an element $j_1'\in \SJ(\A^2_K)$ such that $j_1j_1'\in \SAff(\A^2_K)$ or if it starts with an element $a_1'\in \SAff(\A^2_K)$ such that $(a_1')^{-1}\alpha_m\in \SJ(\A^2_K)$ and finish with a simplified writing where we can directly read that $\deg(f^2)>\deg(f)$.

Assertion $(2)$ corresponds to ask that the set of algebraic elements of $\SAut(\A^2_K)_{\le d}$ is closed, for each $d\ge 1$. This follows from $(1)$.
\end{proof}
\subsection{Conjugacy classes of elements of $\SJ(\A^2_K)$.}It is easy to decide whether two H\'enon automorphisms are conjugate, using the amalgamated product structure. The writing of such an element is unique up to cycling permutation of the elements and up to inserting elements of $\SJ(\A^2_K)\cap \SAff(\A^2_K)$. The classification of conjugacy classes of $\SAut(\A^2_K)$ reduces thus to the study of elements of $\SJ(\A^2_K)$.

\begin{lemma}\label{Lem:Fourfamilies}
Let $f\in\SJ(\A^2_K)$ be some element. After conjugation in this group, the element belongs to one of the following families:
\[\begin{array}{lll}
(i)& (ax_1,a^{-1}x_2),& a\in K\setminus \{0,1\},\\
(ii)&(x_1+P(x_2),x_2),& P\in K[x_2],\\
(iii)&(\zeta x_1+(x_2)^{m-1}P((x_2)^m),\zeta^{-1}x_2), & \zeta\in K\setminus \{0,1\}\mbox{ a primitive $m$-th root of unity},\\
&& P\in K[x_2]\setminus\{0\},\\
(iv) &\left\{\begin{array}{l}
(x_1,x_2+1),\\
(x_1+(x_2)^{p-1}P((x_2)^p),x_2+1)\end{array}\right.& \begin{array}{l}
\car(K)=0,\\
P\in K[x_2],\car(K)=p.\end{array}\end{array}\]
\end{lemma}
\begin{proof}
An element  $f\in \SJ(\A^2_K)$ is equal to $(ax_1+P(x_2),a^{-1}x_2+c)$ for some $a\in K^*$, $c\in K$, $P\in K[x_2]$. 

If $a\not=1$, we conjugate by $\left(x_1,x_2+\frac{ac}{1-a}\right)$ and can assume that $c=0$. Conjugating by $(x_1-\lambda (x_2)^n,x_2)$ yields $(ax_1+P(x_2)+\lambda (a-a^{-n})(x_2)^n,a^{-1}x_2).$
We can then kill the coefficient of degree $n$ of $P$, if $a\not=a^{-n}$. This gives $(i)$ or $(iii)$.

If $a=1$ and $c=0$, we get $(ii)$. If $a=1$ and $c\not=0$, we conjugate by $(cx_1,c^{-1}x_2)$ and can assume that $c=1$. Conjugating $(x_1+P(x_2),x_2+1)$ by $(x_1-\lambda (x_2)^{n+1},x_2)$ yields $(x_1+P(x_2)+\lambda ((x_2)^{n+1}-(x_2+1)^{n+1}),x_2+1).$
We can thus kill the coefficient of degree $n$ if $n+1$ is not a multiple of $\car(K)$. This gives  $(iv)$.
\end{proof}
The most complicate case corresponds to Family $(iv)$, in the case of a field of characteristic $p>0$. In order to describe the conjugacy classes among the family, we will need the following lemma:
\begin{lemma}\label{Lemma:ExactSeqP}
Let $K$ be a field of characteristic $p>0$. Let $V\subset K[x]$ be the subvector space given by $$V=\{x^{p-1}P(x^p)\mid P\in K[x]\}$$ and let $\delta,N$ be the following $K$-linear maps
$$\begin{array}{rrcl}
\delta\colon &K[x]&\to& K[x]\\
& F(x)&\mapsto & F(x+1)-F(x),\\ \\
N\colon &K[x]&\to& K[x]\\
& F(x)&\mapsto & F(x)+F(x+1)+\dots +F(x+p-1).
\end{array}$$
Then, $\Im(\delta)=\Ker(N)$ and $K[x]=V \oplus \Im(\delta)=V \oplus \Ker(N)$.
\end{lemma}
\begin{proof}
$(1)$ We first show that $V+\Im(\delta)=\k[x]$. This is the same argument as what we did in Lemma~\ref{Lem:Fourfamilies}(iv). We take any polynomial $P(x)=\sum_{j=1}^l a_j x^j\in K[x]$. If $P\not\in V$, we can define $m$ to be the biggest integer such that $a_m\not=0$ and $m+1\not\equiv 0\pmod{p}$. Replacing $P$ with $\tilde{P}(x)=P(x)-\delta(\frac{a_m}{m+1} x^{m+1})$, we kill the coefficient of $x^m$. Continuing this process until all coefficients $a_i$ with $i+1\not\equiv 0\pmod{p}$ are zero, we obtain an element of $V$.

$(2)$ We then show that $\Ker(N)\cap V=\{0\}$. Let $F(x)=\sum_{i=1}^n a_i x^{ip-1}$ be an element of $V$, and assume that $a_n\not=0$.
Since 
$$N(F(x))=\sum_{i=1}^n a_i\left( \sum_{k=0}^{p-1} (x+k)^{ip-1}\right),$$ 
the coefficient of $x^{np-p}$ of  $N(F(x))$ is equal to 

$$a_n\left(\begin{array}{c} np-1\\ np-p\end{array}\right)\sum_{k=0}^{p-1} k^{p-1}=a_n\left(\begin{array}{c} np-1\\ np-p\end{array}\right)(p-1)\in K^*.$$
This shows that $F\notin\Ker(N)$. Hence $V\cap \Ker(N)=\{0\}$.

$(3)$ The inclusion $\Im(\delta)\subset \Ker(N)$ follows from the definition of $\delta$ and $N$, so we obtain  $V\cap \Im(\delta)=\{0\}$ by $(2)$. Moreover, since $\Ker(N)$ is contained in $\Im(\delta)\oplus V$ (by $(1)$) and $\Ker(N)\cap V=\{0\}$, we have $\Im(\delta)=\Ker(N)$.\end{proof}
\begin{remark}\label{Remark:Pnotidentity}
The equality $\Ker(N)\cap V=\{0\}$ of Lemma~\ref{Lemma:ExactSeqP} corresponds to the fact that an element of the form $$(x_1+(x_2)^{p-1}P((x_2)^p),x_2+1),$$ where $P\in K[x_2],\car(K)=p$ (Family $(iv)$) is of order $p$ if and only if $P=0$. Indeed, $$(x_1+Q(x_2),x_2+1)^p=(x_1+Q(x_2)+\dots+Q(x_2+p-1),x_2),$$ for each $Q\in K[x_2]$.
\end{remark}
We will also need the following lemma, which is a direct consequence of the amalgamated product structure of $\SAut(\A^2_K)$.
\begin{lemma}\label{Lemm:SJorSAutSame}
Let $f\in \SJ(\A^2_K)$, which is not conjugate to $(ax_1,a^{-1}x_2)$, $a\in K^*$ or to $(x_1+1,x_2)$ or $(x_1,x_2+1)$ in the group $\SJ(\A^2_K)$. Then, $f$ is conjugate to an element of $g\in \SJ(\A^2_K)$ in the group $\SAut(\A^2_K)$ if and only if it is conjugate to $g$ in $\SJ(\A^2_K)$. 
\end{lemma}
\begin{proof}
We suppose that $g=\alpha f \alpha^{-1}\in \SJ(\A^2_K)$, for some $\alpha\in \SAut(\A^2_K)\setminus\SJ(\A^2_K)$. We can write $\alpha$ as $$\alpha=j_{m+1}a_{m}j_m\dots j_2a_1j_1,$$ where $m\ge 1$, each $j_i\in \SJ(\A^2_K)$, each $a_i\in \SAff(\A^2_K)$, and such that $a_i\not\in \SJ(\A^2_K)$ for $i=1,\dots,m$, $j_i\notin\SAff(\A^2_K)$ for $i=2,\dots,m-1$.
Since $\alpha f \alpha^{-1}\in \SJ(\A^2_K)$, the element $j_1 f (j_1)^{-1}$ belongs to $\SAff(\A^2_K)\cap \SJ(\A^2_K)$, and the same holds for the element  $a_1 j_1 f (j_1)^{-1} (a_1)^{-1}$.

The fact that $f_1=j_1 f (j_1)^{-1}$ and $f_2=a_1 j_1 f (j_1)^{-1} (a_1)^{-1}$ belong to  $\SAff(\A^2_K)\cap \SJ(\A^2_K)$ corresponds to the fact that both extend to automorphisms of $\p^2_K$ that fix the point $[0:1:0]$ at infinity. However, $a_1\in \SAff(\A^2_K)\setminus \SJ(\A^2_K)$ so it extends to an automorphism of $\p^2_K$, whose action at the infinity moves the point $[0:1:0]$. This implies that $f_1$ fixes at least two points at infinity. Conjugating by an element of $\SAff(\A^2_K)\cap \SJ(\A^2_K)$ we can assume that these points are $[0:1:0]$ and $[0:0:1]$. This implies that $f_1=(ax_1+b,a^{-1}x_2+c)$, for some $(a,b,c)\in K^*\times K^2$. It remains to see that $f_1$ is conjugate to  $(ax_1,a^{-1}x_2)$, $(x_1+1,x_2)$ or $(x_1,x_2+1)$, in the group $\SJ(\A^2_K)$. This is a straight-forward calculation, already done in the proof of Lemma~\ref{Lem:Fourfamilies}.
\end{proof}

\begin{proposition}[Conjugacy classes of algebraic elements in $\SAut(\A^2_K)$]\label{Prop:ConjclassesSJ}Every algebraic element of $\SAut(\A^2_K)$ is conjugate to one of the following families $(i)$--$(iv)$ of Lemma~$\ref{Lem:Fourfamilies}$.

Apart from $(x_1,x_2+1)$, which is conjugate to $(x_1+1,x_2)$, no two elements of distinct families are conjugate in $\SAut(\A^2_K)$. Moreover, the conjugacy classes in $\SAut(\A^2_K)$ among the families are given as follows:

$(i)$ $(ax_1,a^{-1}x_1)\sim (bx_1,b^{-1}x_2)$  $\Leftrightarrow$ $a=b^{\pm 1}$.

$(ii)$ $(x_1+P(x_2),x_2)\sim (x_1+Q(x_2),x_2)$ $\Leftrightarrow$ $Q(x_2)=aP(ax_2+b)$ for some $(a,b)\in K^*\times K.$

$(iii)$ $(\zeta x_1+(x_2)^{m-1}P((x_2)^m),\zeta^{-1}x_2)$ is conjugate to $(\zeta x_1+(x_2)^{m-1}a^mP((ax_2)^m),\zeta^{-1}x_2)$ for each $a\in K^*$, but not to any other element of Family $(iii)$.

$(iv)$ Two elements $$f=(x_1+(x_2)^{p-1}P((x_2)^p),x_2+1),\ g=(x_1+(x_2)^{p-1}Q((x_2)^p),x_2+1),$$ where $P,Q\in K[x_2]$, are conjugate if and only their $p$-th powers
$$f^p=(x_1+\tilde{P}(x_2),x_2), g^p=(x_1+\tilde{Q}(x_2),x_2),$$
where $\tilde{P},\tilde{Q}\in K[x_2]$, satisfy $\tilde{Q}(x_2)=\tilde{P}(x_2+c)$ for some $c\in K$.
\end{proposition}
\begin{proof}The first assertion follows from Lemma~\ref{Lem:Henon} (see Remark~\ref{Rem:AlgebraicSJ}). It remains to describe the conjugacy classes between the families $(i)$--$(iv)$.
We first observe that if $\alpha\in \SL(2,K)$ is conjugate by $\nu\in \SAut(\A^2_K)$ to an element $\beta\in \SAut(\A^2_K)$ that fix the origin, the derivative of the equation $\alpha\nu =\nu\beta$ at the origin yields $\alpha D_\nu(0)=D_\nu(0)D_\beta(0)$, so the derivative of $\beta$ at the origin is conjugated to $\alpha$ in $\SL(2,K)$. This shows that the families $(i)$ and $(ii)$ are disjoint, and gives the conjugacy classes among $(i)$.

We now observe that two elements $f=(x_1+P(x_2),x_2)$ and $(x_1+Q(x_2),x_2)$ are conjugate in $\SAut(\A^2_K)$ if and only if they are conjugate in $\SJ(\A^2_K)$. This is given by Lemma~\ref{Lemm:SJorSAutSame}, if $f$ or $g$ is not conjugate to $(x_1+1,x_2)$ or $(x_1,x_2)$ in $\SJ(\A^2_K)$, and is trivial in the remaining cases (when both $f$ and $g$ are conjugate to one of these two elements). It remains to observe that the conjugation of $f=(x_1+P(x_2),x_2)$ by $(ax_1+R(x_2),a^{-1}(x_2-b))$ yields $(x_1+aP(ax_2+b),x_2)$ to obtain the conjugacy classes among~$(ii)$.

An element $f=(\zeta x_1+(x_2)^{m-1}P((x_2)^m),\zeta^{-1}x_2)$ in Family $(iii)$ satisfies $f^m=(x_1+m(x_2)^{m-1}P((x_2)^m),x_2)\not=(x_1,x_2)$. Since $f^m$ belongs to Family $(ii)$, $f^m$ (and thus $f$) is not conjugate to an element of Family $(i)$. Moreover, $f$ is not conjugate to an element of Family $(ii)$ because $\zeta^{-1}\in K\setminus \{0,1\}$ is an eigenvalue of the action of $f^*$ on $K[x_1,x_2]$. By Lemma~\ref{Lemm:SJorSAutSame}, $f$ is conjugate to another element $g$ of Family $(iii)$ in $\SAut(\A^2_K)$ if and only if this conjugation holds in $\SJ(\A^2_K)$. Looking at the second coordinate, we have to conjugate by $\alpha=(ax_1+R(x_2),a^{-1}x_2)$. This implies that 
$g=(\zeta x_1+(x_2)^{m-1}Q((x_2)^m),\zeta^{-1}x_2)$ for some polynomial $Q$. Looking at $g^m=\alpha f^m \alpha^{-1}=(x_1+ma^m(x_2)^{m-1}P((ax_2)^m),x_2),$
we obtain $Q((x_2)^m)=a^m P((ax_2)^m)$. This yields a necessary condition on $Q$, which is also sufficient, by choosing $R=0$.

It remains to study the last class $(iv)$. The element  $(x_1,x_2+1)$ is conjugate to $(x_1+1,x_2)$ by $(x_2,-x_1)$. If $\car(K)=0$, there is no other element in the family. So we assume that $\car(K)=p>0$, and consider $f=(x_1+(x_2)^{p-1}P((x_2)^p),x_2+1)$, for some polynomial $P\in K[x_2]\setminus \{0\}$. Because the action of $f^*$ on $K[x_1,x_2]$ has only eigenvalues equal to $1$, $f$ is not conjugate to any element of Family $(i)$ or $(iii)$. Moreover, the fact that $P\not=0$ implies that $f^p$ is not the identity (see Remark~\ref{Remark:Pnotidentity}). Hence, $f$ is not conjugate to an element of Family $(ii)$. 

We then take $g=(x_1+(x_2)^{p-1}Q((x_2)^p),x_2+1)$, for some $Q\in K[x_2]$, write 
$$f^p=(x_1+\tilde{P}(x_2),x_2), g^p=(x_1+\tilde{Q}(x_2),x_2),$$
for some $\tilde{P},\tilde{Q}\in K[x_2]$, and show that $f$ and $g$ are conjugate if and only if $\tilde{Q}(x_2)=\tilde{P}(x_2+c)$ for some $c\in K$. This will conclude the proof.

Suppose first that $f$ is conjugate to $g$. Lemma~\ref{Lemm:SJorSAutSame} yields the existence of $\alpha\in \SJ(\A^2_K)$ such that $\alpha f\alpha^{-1}=g$. Looking at the second coordinate, we see that $\alpha=(x_1+R(x_2),x_2-c)$ for some $R\in K[x_2],c\in K$. Then,
$\alpha f^p \alpha^{-1}=(x_1+\tilde{P}(x_2+c),x_2)$, which implies that $\tilde{Q}(x_2)=\tilde{P}(x_2+c)$ as we wanted.

Conversely, suppose that $\tilde{Q}(x_2)=\tilde{P}(x_2+c)$ for some $c\in K$. We conjugate then $g$ with $(x_1,x_2+c)$; this changes maybe the polynomial $Q$, and replaces $\tilde{Q}$ with $\tilde{P}$. Hence, we can assume that $f^p=g^p$. This corresponds to the equality $$N((x_2)^{p-1}P((x_2)^p))=N((x_2)^{p-1}Q((x_2)^p)),$$
where $N\colon K[x_2]\to K[x_2]$ is the $K$-linear map that sends $f(x_2)$ onto $f(x_2)+f(x_2+1)+\dots +f(x_2+p-1)$. The element $(x_2)^{p-1}Q((x_2)^p)-(x_2)^{p-1}P((x_2)^p)$ is an element of the kernel of $N$, and is thus equal to $R(x_2+1)-R(x_2)$ for some polynomial $R\in K[x_2]$  (Lemma~\ref{Lemma:ExactSeqP}). This corresponds to the fact that $f$ is conjugate to $g$ by $\alpha=(x_1+R(x_2),x_2)$.
\end{proof}

\subsection{Degenerations between the four families of conjugacy classes}
In the remaining part of the article, we show that every algebraic element of $\SAut(\A^2_K)$ admits a degeneration of conjugates towards a diagonalisable element.

\subsubsection{From Family $(ii)$ and $(iii)$ to Family $(i)$}\label{SubSecFamilyiiandiiitoi}
These are the easiest degenerations. Taking any $f\in \SAut(\A^2_K)$ that belongs to Family $(ii)$ or $(iii)$, and conjugating by diagonal elements of $\SAut(\A^2_{K(t)})$ we obtain elements of $\SAut(\A^2_{K[t]})$ which yield families of conjugates of $f$  that converge towards an element of Family $(i)$. Indeed, we have 
$$\begin{array}{rcl}
(tx_1,t^{-1}x_2)(x_1+P(x_2),x_2)(t^{-1}x_1,tx_2)&=&(x_1+tP(x_2),x_2),\\
(tx_1,t^{-1}x_2)f(t^{-1}x_1,tx_2)&=&(\zeta x_1+t^{m}(x_2)^{m-1}P(t(x_2)^m),\zeta^{-1}x_2).\end{array}$$
In particular, there is a diagonalisable element in the closure of the conjugacy class of any element of Families $(ii)$ and $(iii)$, and thus of any algebraic element of $\SAut(\A^2_K)$, if $\car(K)=0$.
\subsubsection{Family $(iv)$}
The last family, when $\car(K)=p>0$, is the most interesting. The elements of this family are of the form $(x_1+Q(x_2),x_2+1)$, for some polynomial $Q\in K[x_2]$. Diagonal conjugations by $(tx_1,t^{-1}x_2)$ or $(t^{-1}x_1,tx_2)$ yield elements of $\SAut(\A^2_{K[t^{\pm1}]})$ which do not have a value at $t=0$.

In the following proposition, we provide two explicit degenerations, typical to positive characteristic, to either $(x_1,x_2+1)$ or the identity.
\begin{proposition}\label{Prop:IVtoI}
Assume that $\car(K)=p>0$ and let 
$$f=(x_1+Q(x_2),x_2+1)$$
for some polynomial $Q\in K[x_2]$.

Then, there are two elements $F_1,F_2\in \SAut(\A^2_{K[t]})$, conjugate to $f$ in $\SAut(\A^2_{K[t^{\pm 1}]})$, which have the following properties.
\begin{enumerate}[$(1)$]
\item
For each $t\in K\setminus \{0\}$, $F_1(t),F_2(t)\in \SAut(\A^2_K)$ are conjugate to $f$ in $\SAut(\A^2_K)$.
\item
$F_1(0)=(x_1,x_2+1)$, $F_2(0)=(x_1,x_2).$\end{enumerate}
\end{proposition}
\begin{proof}If $Q=0$, we can choose $F_1(t)=f$ and $F_2(t)$ to be the conjugation of $f$ by $(t^{-1}x_1,tx_2)\in \SAut(\A^2_{K[t^{\pm 1}]})$.

We can thus assume that $Q\not=0$, denote by $d\ge 0$ the degree of $Q$ and by $\mu\in K^{*}$ the coefficient of degree $d$. We then choose some integer $a\ge 1$ such that $q=p^a>d$, write $\lambda=\frac{1}{\mu^q}\in K^*$, and define
$$\alpha=\left(\frac{x_1}{t^d},t^dx_2+\lambda x_1^{q}+\frac{1}{t}\right)\in \SAut(\A^2_{K[t,t^{-1}]}).$$
A direct computation yields
$$\alpha^{-1} f\alpha=
\left(x_1+t^dQ(t^dx_2+\lambda x_1^q+\frac{1}{t}),x_2+\frac{1}{t^d}+\frac{\lambda x_1^q}{t^d}-\frac{\lambda}{t^d}\left(x_1+t^dQ(t^dx_2+\lambda x_1^q+\frac{1}{t})\right)^q\right).$$
The definition of $d$ and $\mu$ imply that we can write
$$Q(t^dx_2+x_1^q+\frac{1}{t})=\frac{\mu}{t^d}+\frac{P}{t^{d-1}},$$ for some $P\in K[x_1,x_2,t]$. This yields (remembering that $\lambda\mu^q=1$), the following equality
$$\alpha^{-1} f\alpha=\left(x_1+\mu+tP,x_2+\frac{1}{t^d}-\frac{\lambda}{t^d}(\mu +tP)^q\right)=\left(x_1+\mu+tP,x_2-\lambda t^{q-d}P^q\right).$$
Because $q>d$, the value of this element of $\SAut(\A^2_{K[t]})$ at $t=0$ is $\left(x_1+\mu,x_2\right).$ Conjugating $\alpha^{-1} f\alpha$ with $(-\mu x_2,\mu^{-1}x_1)$ yields thus $F_1$.

In order to get $F_2$, we recall that $(t x_1,t^{-1} x_2)\left(x_1+\mu,x_2\right)(t^{-1} x_1,t x_2)=(x_1+t\mu,x_2).$ We are then tempted to use $(t x_1,t^{-1} x_2)\alpha^{-1} f\alpha(t^{-1} x_1,t x_2),$ but this element has in general no value at $t=0$. We then choose $m>0$ and define $\beta$ to be
$$\beta=\left(\frac{x_1}{t^{md}},t^{md}x_2+\lambda x_1^{q}+\frac{1}{t^m}\right)\in \SAut(\A^2_{K[t,t^{-1}]}).$$
Since $\beta$ is obtain from $\alpha$ by replacing $t$ with $t^m$, we obtain 
$$\beta^{-1} f\beta=\left(x_1+\mu+t^mP(x_1,x_2,t^m),x_2-\lambda t^{m(q-d)}P(x_1,x_2,t^m)^q\right).$$
We can now define $F_2(t)=(t x_1,t^{-1} x_2)\beta^{-1} f\beta(t^{-1} x_1,t x_2),$ which is equal to
$$F_2(t)=\left(x_1+t\mu+t^{m+1}P(t^{-1}x_1,tx_2,t^m),x_2-\lambda t^{m(q-d)-1}P(t^{-1}x_1,tx_2,t^m)^q\right).$$
Choosing $m$ big enough, $F_2$ is defined at $t=0$ and $F_2(t)=(x_1,x_2)$.
\end{proof}\subsection{The diagonalisable elements}
In order to finish the proof of Theorem~$\ref{Thm:Dimension2}$ it remains to show that the conjugacy classes of diagonalisable elements is closed. This was shown in \cite{MauFur}, for the group $\Aut(\A^2_\C)$, but with transcendental methods. In the case of $\SAut(\A^2_\k)$, we can however give a simple proof, that works for any algebraically closed field $\k$.

The proof use the following result, that follows from the amalgamated structure product.
\begin{lemma}\label{Lem:ConjugateSmalldegree}Let $K$ be any field, let $\mu\in K^{*}$, $f=(\mu x_1,\mu^{-1} x_2)\in \SAut(\A^2_K)$, and $g\in \SAut(\A^2_K)$ be a conjugate of $f$ in this group. 
Then, there exists  $\alpha\in \SAut(\A^2_K)$ such that $$g=\alpha^{-1} f \alpha\mbox{ and }\deg \alpha \le \deg g.$$
\end{lemma}
\begin{proof}
We write $g=\alpha^{-1} f\alpha$ for some $\alpha\in \SAut(\A^2_K)$, and write $\alpha$ as a product of elements of $\SAff(\A^2_K)$ and $\SJ(\A^2_K)$, in a reduced way (no two consecutive elements belong to the same group). 
Note that $f$ is conjugate to $f^{-1}$ by $(x_2,-x_1)$, so we can easily exchange $f$ with $f^{-1}$ if needed.

If $\alpha$  starts with an element $a\in \SAff(\A^2_K)$, then $a'=a^{-1}ja\in \SAff(\A^2_K)$. If $a'$ does not belong to $\SJ(\A^2_K)$, then 
$\deg(g)=\deg(\alpha)\deg(\alpha^{-1})=\deg(\alpha)^2$ (the degree is the product of the Jonqui\`eres elements occuring, see Remark~\ref{Rem:Henon}). If $a'$ belongs to $\SJ(\A^2_K)$, then it is conjugate to $f$ or $f^{-1}$ in $\SJ(\A^2_K)$ (Proposition~\ref{Prop:ConjclassesSJ}), we can thus replace $a$ with an element of $\SJ(\A^2_K)$ and either decrease the length of $\alpha$ or reduces to the case $\alpha\in \SJ(\A^2_K)$.

Suppose now that $\alpha$ starts with an element $j\in \SJ(\A^2_K)$. Replacing $j$ with $\rho j$, where $\rho$ is diagonal, we can assume that $j=(x_1+P(x_2),x_2+c)$, for some $P\in K[x_2]$ and $c\in K$, we obtain thus $$j'=j^{-1} fj=(\mu x_1+P(x_2)\mu-P(x_2+c(\mu^{-1}-1)),\mu^{-1} x_2+c(\mu^{-1}-1))\in \SJ(\A^2_K).$$
We have thus  $\deg(P(x_2))\ge P(x_2)\mu-P(x_2+c(\mu^{-1}-1))$, and if equality does not hold we can kill the highest coefficient of $P$ without changing $j'=j^{-1} f j$. We reduce then to the case where $\deg j'=\deg j$. If this degree is $1$, then $j\in \SAff(\A^2_K)$ and we can reduce the length of $\alpha$, or obtain $\deg(\alpha)=1$. The remaining case is when $\deg j'=\deg j$. The result is then obtained if $\alpha=j$. Otherwise, $\alpha=ja_1j_1\dots$ is a reduced writing, as well as $g=\dots (j_1)^{-1}(a_1)^{-1} j' a_1j_1\dots$. We again find $\deg g\ge \deg \alpha$.\end{proof}
\begin{proposition}\label{Prp:Conjclassdiagclosed}
Let $\mu\in \k^{*}$, $f=(\mu x_1,\mu^{-1} x_2)\in \SAut(\A^2_\k)$ and $d\ge 1$. The set $$C(f)=\{gfg^{-1}\mid g\in \Aut(\A^2_\k)\}=\{gfg^{-1}\mid g\in \SAut(\A^2_\k)\}$$
is closed in $\SAut(\A^2_\k)$.
\end{proposition}
\begin{proof}
The equality between the two sets is obvious: if $g\in \Aut(\A^2_\k)$, we have $g'=g\tau\in \SAut(\A^2_\k)$ for some diagonal $\tau\in \Aut(\A^2_\k)$, so $gfg^{-1}=g'f{g'}^{-1}$. We can moreover assume that $\mu\not=1$, otherwise the result is trivial.
We fix some integer $d\ge 1$, write $Z=\SAut(\A^2_\k)_{\le d}$ and will show that $C(f)\cap Z$ is closed in $Z$ (this gives the result by Lemma~\ref{Lem:Shafa}). 

We consider the morphism of algebraic varieties $Z\to \Aut(\A^2_\k)_{\le d^2}$  given by $g\mapsto gfg^{-1}$ and denote by $Y\subset \SAut(\A^2_\k)_{\le d}$ the preimage of $Z$. This gives rise to a morphism $\varphi \colon Y\to Z$, whose image is equal to $C(f)\cap Z$ by Lemma~\ref{Lem:ConjugateSmalldegree}.  By Corollary~\ref{Coro:LimitConjugateValuation}, it suffices to take an irreducible $\k$-curve $\Gamma$, a rational map $\iota\colon \Gamma\dasharrow Y$ such that $\hat{\varphi}=\varphi\circ \iota\colon \Gamma\dasharrow Z$ is defined at a smooth point $p\in \Gamma$, and to show that the image $g=\hat{\varphi}(p)\in Z$ belongs to $C(f)$.

Note that $\iota$ corresponds to an element $\alpha\in \Aut(\A^2_{\k(\Gamma)})$, and that $\hat{\varphi}$ to the element $\alpha f \alpha^{-1}\in \Aut(\A^2_{\k(\Gamma)})$, which is defined at $p$ and whose value at $p$ corresponds to $g$. The set of algebraic elements being closed (Corollary~\ref{Coro:Algclosed}), the element $g$ is conjugate to an element of $\SJ(\A^2_\k)$ (Remark~\ref{Rem:AlgebraicSJ}). Looking at the actions on $\k(\Gamma)[x_1,x_2]$, we find that $f^{*}(x_1)=\mu x_1$ so $(\alpha f \alpha^{-1})^*(v)=\mu v$, where $v=(\alpha^{-1})^*(x_1)\in \k(\Gamma)[x_1,x_2]$. Replacing $v$ with $\lambda v$, where $\lambda \in \k(\Gamma)$, does not change the equation $(\alpha f \alpha^{-1})^*(v)=\mu v$. We can therefore assume that $\alpha\in \mathcal{O}_p(\Gamma)[x_1,x_2]$ and that its value at $p$ is not zero. This shows that $\mu$ is also an eigenvalue of $g^*$. By Lemma~\ref{Lem:Fourfamilies}, $g$ is conjugate in $\SJ(\A^2_\k)$ to $(\nu x_1,\nu^{-1}x_2)$, for some $\nu\in \k\setminus \{0,1\}$, or to $(\zeta x_1+(x_2)^{m-1}P((x_2)^m),\zeta^{-1}x_2)$ for some primitive $m$-th root of unity $\zeta$, and $P\in \k[x_2]\setminus\{0\}$, and $\mu$ belongs to the subgroup of $(\k^*,\cdot)$ generated by $\nu$ (respectively of $\zeta$). Let us observe that the second case is not possible. Indeed, otherwise we would have $f^m=\mathrm{id}$, so $(\alpha f \alpha^{-1})^m=\mathrm{id}$ and thus $g^m=\mathrm{id}$, which is not the case since $P\not=0$. We can thus assume, after composing $\alpha$ with an element of $\SAut(\A^2_\k)$, that $g=(\nu x_1,\nu^{-1}x_2)$ for some $\nu \in \k\setminus \{0,1\}$.

Denote by $U\subset C$ the dense open subset where $\iota$ is defined. Replacing $C$ with $U\cup \{p\}$, we can assume that $U=C\setminus \{p\}$ (if $p\in U$, the result is obvious), and thus that $\hat{\varphi}$ is a morphism $\hat{\varphi}\colon C\to Z$. Denote by $\kappa\colon C\dasharrow \A^2_\k$ the rational map given by $\kappa(u)=\iota(u)(0,0)$, when $u\in U$ (the image of the origin of~$\A^2_\k$). 

We now prove that $\kappa$ is defined at $p$. To do this, we define $F\subset C\times \A^2_\k$ to be the closed set $F=\{(c,(x_1,x_2))\mid \hat{\varphi}(c)(x_1,x_2)=(x_1,x_2)\}$. Since $F$ is given by only two equations, each irreducible component of $F$ has dimension at least $1$. The intersection of $F$ with $u\times \A^2_\k$ yields one point, for each $u\in C$, hence the projection to $C$ yields an isomorphism $\pi\colon F\to C$. The rational map $C\dasharrow F$ given by $u\mapsto (u,\kappa(u))$ corresponding to the identity, it is thus defined at $p$.

The map $(x_1,x_2)\mapsto (x_1,x_2)+\kappa(u)$ corresponds therefore to an element of $\Aut(\A^2_{\mathcal{O}(C)})$, and replacing $\alpha$ with its composition by this one, we can assume that $\alpha$ fixes the origin. We obtain thus that the derivative of $\alpha f \alpha^{-1}$ at the origin is conjugate to $f$ by an element of $\SL(2,\k)$, for each $u\in U$, so its eigenvalues are $\mu$ and $\mu^{-1}$, for any point of $U$. At the point $p$, we obtain $g=(\nu x_1, \nu^{-1} x_2)$, which implies that $\nu=\mu^{\pm 1}$, so that $g$ and $f$ are conjugate in $\SAut(\A^2_\k)$ as we wanted.\end{proof}

\begin{proof}[Proof of Theorem~$\ref{Thm:Dimension2}$]
By Lemma~\ref{Lem:Henon} and Remarks~\ref{Rem:Henon} and~\ref{Rem:AlgebraicSJ}, an algebraic element of $\SAut(\A^2_\k)$ is conjugated to an element of $\SJ(\A^2_\k)$ and a non-algebraic element is conjugate to a H\'enon automorphism, which is dynamically regular. In the latter case, the conjugacy class is closed by  Theorem~\ref{Thm:Regular}, this gives $(3)$.

If $f\in \SAut(\A^2_\k)$ is diagonalisable, its conjugacy class is closed by  Proposition~\ref{Prp:Conjclassdiagclosed}.

If $f\in \SAut(\A^2_\k)$ is algebraic but not diagonalisable, it is conjugate to an element of the Families $(ii)$, $(iii)$ or $(iv)$ of Lemma~\ref{Lem:Fourfamilies}. The existence of the degeneration stated in $(2)$ is given in $\S\ref{SubSecFamilyiiandiiitoi}$ for families $(ii)$, $(iii)$ and in Proposition~\ref{Prop:IVtoI} for Family $(iv)$.
\end{proof}
\end{document}